\documentclass[10pt]{article}
\usepackage{geometry} 
\pdfoutput=1
\usepackage{graphicx,url}
\usepackage{amssymb,amsmath, amsthm}
\usepackage{epstopdf,setspace}
\author{ADG}
\allowdisplaybreaks

\usepackage[T1]{fontenc}
\newtheorem{theorem}{Theorem}
\newtheorem{remark}{Remark}[section]

\title{On random flights with non-uniformly distributed directions}
\author{ Alessandro De Gregorio\\
\large Dipartimento di Scienze Statistiche \\
\large``Sapienza'' University of Rome\\\
\large P.le Aldo Moro, 5 - 00185, Rome - Italy \\
\large alessandro.degregorio@uniroma1.it}
\numberwithin{equation}{section}
\begin{document}

\maketitle

\begin{abstract}

This paper deals with a new class of random flights $\underline{\bf X}_d(t),t>0,$ defined in the real space $\mathbb{R}^d, d\geq 2,$ characterized by non-uniform probability distributions on the multidimensional sphere. These random motions differ from similar models appeared in literature which take directions  according to the uniform law. The  family of angular probability distributions introduced in this paper depends on a parameter $\nu\geq 0$ which gives the level of drift of the motion. Furthermore, we assume that the number of changes of direction performed by the random flight is fixed. The time lengths between two consecutive changes of orientation have joint  probability distribution given by a Dirichlet density function.  

The analysis of $\underline{\bf X}_d(t),t>0,$ is not an easy task, because it involves the calculation of integrals which are not always solvable. Therefore, we analyze the random flight $\underline{\bf X}_m^d(t),t>0,$ obtained as projection onto the lower spaces $\mathbb{R}^m,m<d,$ of the original random motion in $\mathbb{R}^d$. Then we get the probability distribution of $\underline{\bf X}_m^d(t),t>0.$ 

Although, in its general framework,  the analysis  of $\underline{\bf X}_d(t),t>0,$ is very complicated, for some values of $\nu$, we can provide some results on the process. Indeed, for $\nu=1$, we obtain the characteristic function of the random flight moving in $\mathbb{R}^d$. Furthermore, by inverting the characteristic function, we are able to give the analytic form (up to some constants) of the probability distribution of $\underline{\bf X}_d(t),t>0.$ 
\\

 
\end{abstract}

\section{Introduction}
The random flights have been introduced for describing the real motions with finite speed. The original formulation of the random flight problem is due to Pearson, which considers a random walk with fixed and constant steps. Indeed, the Pearson's model deals with a random walker moving in the plane in straight lines with fixed length and turning through any angle whatever. The main object of interest is the position reached by the random walker after a fixed number of the steps.

 Over the years many researchers have proposed generalizations of the previous Pearson's walk randomizing the spatial displacements. In particular, the random flights have been analyzed independently by several authors starting from the same two main assumptions. The first one concerns the directions which are supposed uniformly distributed on the sphere. Furthermore, the length of the intervals between two consecutive changes of direction is an exponential random variable. Therefore, there exists an underlying homogeneous Poisson process governing the changes of direction. The reader can consult, for instance, Stadje (1987), Masoliver {\it et al}. (1993), Franceschetti (2007), Orsingher and De Gregorio (2007), Garcia-Pelayo (2008). A planar random flight with random time has also been studied in Beghin and Orsingher (2009).
 
Exponential times are not the best choice for many important applications in physics, biology, and engineering, since they assign high probability mass to short intervals. For this reason, recently, the random flight problem has been tackled modifying the assumption on the exponential intertimes. For example, Beghin and Orsingher (2010) introduced a random motion which changes direction at even-valued Poisson events. In other words this model assumes that the time between successive deviations is a Gamma random variable. It can also be interpreted as the motion of particles that can hazardously collide with obstacles of different size, some of which are capable of deviating the motion. Very recently, multidimensional random walks with Gamma intertimes have been also taken into account by  Le Ca\"er (2011) and Pogorui and Rodriguez-Dagnino (2011). Le Ca\"er (2010) and De Gregorio and Orsingher (2011), considered the joint distribution of the time displacements as Dirichlet random variables with parameters depending on the space in which the walker performs its motion. By means of the Dirichlet law we are able to assign higher probability to time displacement with intermediate length. The Dirichlet density function in somehow generalizes the uniform law and permits to explicit for each space the exact distribution of the position reached by the motion at time $t$.

In the cited papers, it is assumed that the directions are uniformly distributed on the sphere. This last assumption is not completely appropriate. Indeed, the real motions are persistent and tend to move along the same direction. Therefore, it seems not exactly realistic to have the same probability to move along each direction of the space when the motion chooses a new orientation. The aim of this paper is to analyze a spherically asymmetric random motion in $\mathbb{R}^d$, that is with non-uniformly distributed directions. Essentially, the literature is devoted to analyze random walks with uniform distribution, while on the random flights with non-uniform distributed angles few references exist: see for instance Grosjean (1953) and Barber (1993). For this reason we believe that this topic is interesting and merits further investigation.

For the sake of completeness, we point out that other multidimensional random models with finite velocity, different from the random flights, have been proposed in literature: see, for example, Orsingher (2000), Samoilenko (2001), Di Crescenzo (2002), Lachal (2006) and Lachal {\it et al}. (2006).

We describe the random flight model analyzed in this paper. Let us consider a particle which starts from the origin of $\mathbb{R}^d,d\geq 2$, and performs its motion with a constant velocity $c>0$. Hereafter, we assume that in the time interval $[0,t]$, a number $n$, with $n\geq 1$, of changes of direction of the motion is recorded. We suppose that the instants at which the random walker changes direction are $0<t_1<t_2<\cdots<t_n<t$, with $t_0=0,\,t_{n+1}=t$, and denote the length of time separating these instants by $\tau_j=t_j-t_{j-1}$, $1\leq j\leq n+1$ where $0<\tau_j<t-\sum_{k=0}^{j-1}\tau_k$, $1\leq j\leq n$, and $\tau_{n+1}=t-\sum_{j=1}^n\tau_j$. 

The directions of the particle are represented by the points on the surface of the $d$-dimensional hypersphere with radius one.  We assume that the random vector $\underline\theta=(\theta_1,...,\theta_{d-2},\phi)$, representing the orientation of the particle, has density function given by

 \begin{figure}[t]
  \begin{center}
\includegraphics[angle=0,width=1\textwidth]{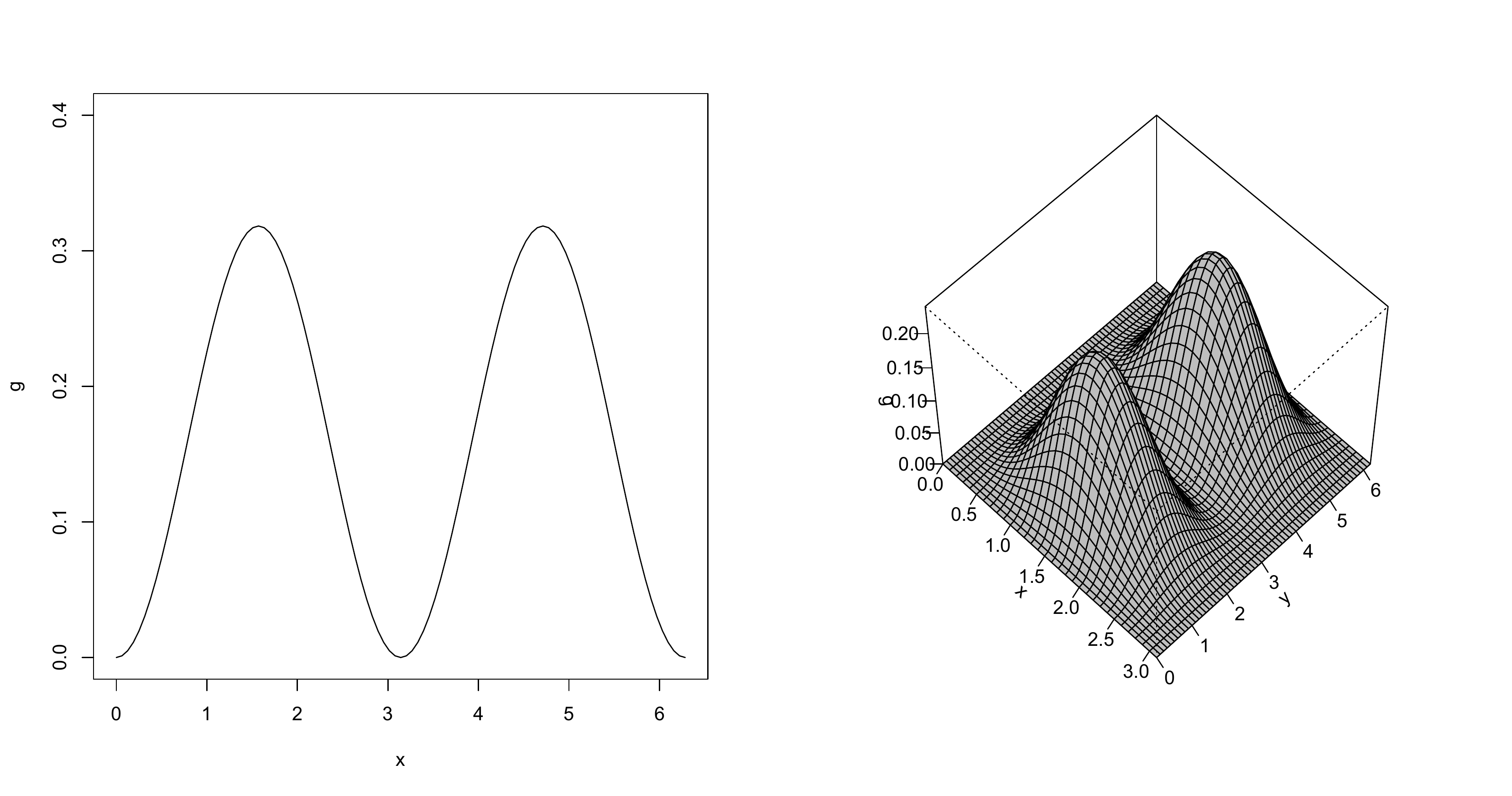}
\caption{The picture on the left represents the behavior of density function $g_{2,1}(x)=\frac1\pi\sin^2x$ with $x\in[0,2\pi]$. On the right it is depicted the density  $g_{3,1}(x,y)=\frac{3}{4\pi}\sin^3x\sin^2y$ with $x\in[0,\pi]$ and $y\in[0,2\pi]$.}\label{plotdens2}
\end{center}
\end{figure}
 
 \begin{figure}[t]
 \begin{center}
\includegraphics[angle=0,width=0.5\textwidth]{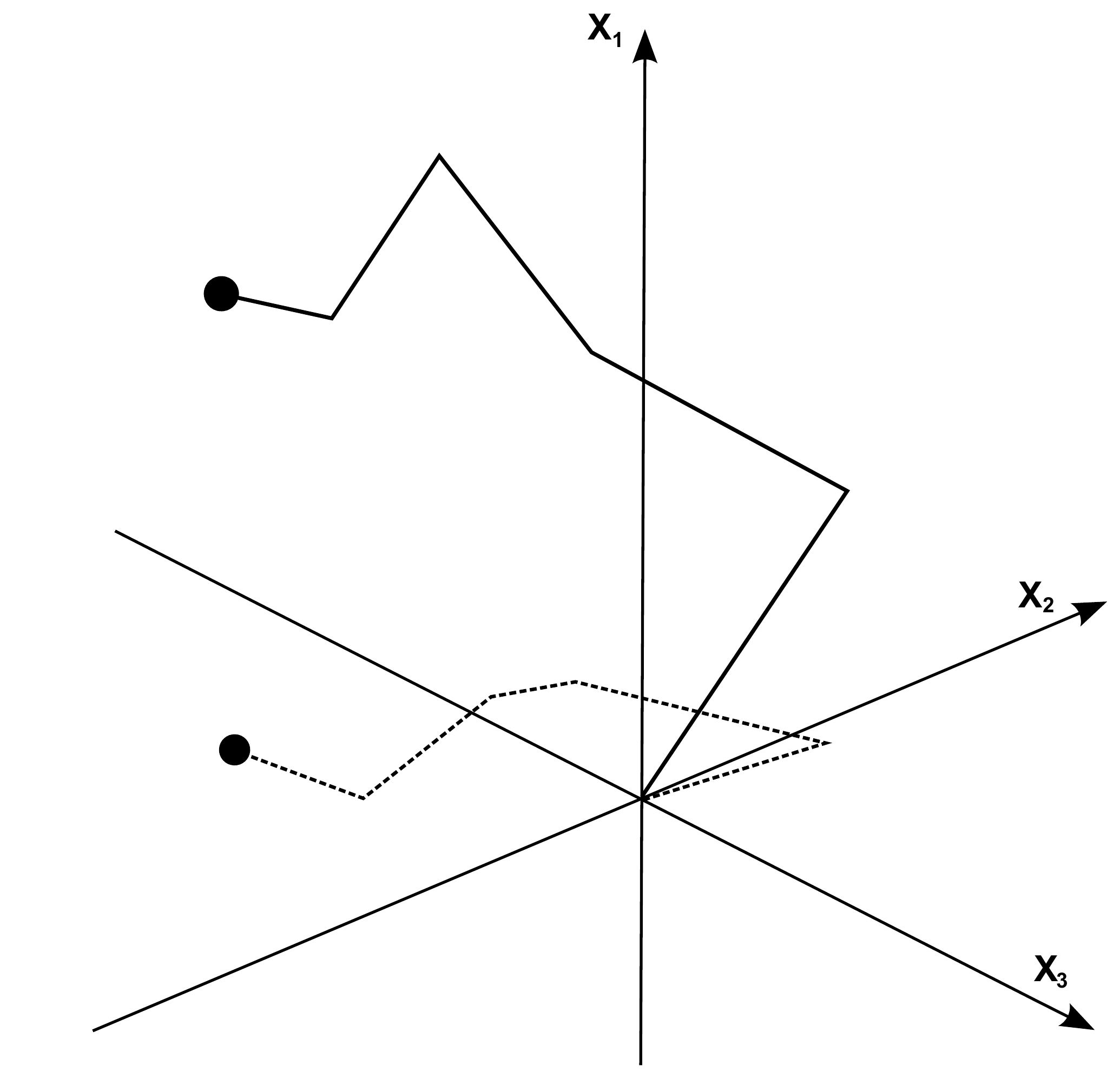}
\caption{A sample path of the three-dimensional random flight with its projection onto the plane.}
\end{center}\label{plotpath}
\end{figure}

\begin{equation}\label{eq:densdir}
g_{d,\nu}(\theta_1,\theta_2,...,\theta_{d-2},\phi)=\frac{\Gamma(\nu+d/2)}{2\pi^{\frac{d-1}{2}}\Gamma(\nu+1/2)}\sin^{2\nu+d-2}\theta_{1}\sin^{2\nu+d-3}\theta_{2}\cdot\cdot\cdot\sin^{2\nu+1}\theta_{d-2}\sin^{2\nu}\phi
\end{equation}
with $\theta_j\in[0,\pi],j=1,2,...,d-2,$ $\phi\in[0,2\pi]$ and $\nu\geq 0$. Hence, the particle chooses the direction not uniformly on the sphere, but following an asymmetric density law. Furthermore, the particle chooses a new direction independently from the previous one. In Figure \ref{plotdens2} we have depicted the behavior of \eqref{eq:densdir} for $\nu=1$ and $d=2,3$. From the picture on the left we can observe that the function \eqref{eq:densdir} is bimodal and assigns high probability mass near the points $\frac \pi2$ and $\frac32\pi$. Similar considerations hold for the picture related to $d=3$. Clearly when $\nu$ increases the function $g$ tends to concentrate around the points $\frac \pi2$ and $\frac32\pi$. For $\nu=0$ we obtain the uniform distribution over the $d$-dimensional hypersphere, that is
\begin{equation*}
g_{d,0}(\theta_1,\theta_2,...,\theta_{d-2},\phi)=\frac{\Gamma(d/2)}{2\pi^{\frac{d}{2}}}\sin\theta_{1}^{d-2}\sin\theta_{2}^{d-3}\cdot\cdot\cdot\sin\theta_{d-2}.
\end{equation*}
Another important assumption is the following: the vector $\underline\tau=(\tau_1,\tau_2,...,\tau_n)$ has distribution given by
\begin{equation}\label{eq:timedis}
f_{d,\nu}(\tau_1,...,\tau_n)=\frac{\Gamma((n+1)(2\nu+d-1))}{(\Gamma(2\nu+d-1))^{n+1}}\frac{\prod_{k=1}^{n+1}\tau_k^{2\nu+d-2}}{t^{(n+1)(2\nu+d-1)-1}}.
\end{equation}
which represents a rescaled Dirichlet random variable, with parameters $(2\nu+d-1,...,2\nu+d-1),\, d\geq 2$.

Let us denote with  $\underline{\bf X}_d(t),t>0,$ the process representing the position reached, at time $t$, by the particle following the random rules above described. In the rest of the paper we are going to analyze the random flights defined by  $\underline{\bf X}_d(t),t>0,$ and their probabilistic characteristics represent the main object of our analysis. Recalling that the motion develops at constant velocity $c$, and exploiting the hyperspherical coordinates, the $d$-dimensional random flight $\underline{\bf X}_d(t)=(X_1(t),...,X_d(t)),t>0,$ has components equal to
\begin{align}\label{polar}
&X_d(t)=c\sum_{k=1}^{n+1}\tau_k
\sin\theta_{1,k}\sin\theta_{2,k}\cdot\cdot\cdot\sin\theta_{d-2,k}\sin\phi_{k}\notag\\
&X_{d-1}(t)=c\sum_{k=1}^{n+1}\tau_k
\sin\theta_{1,k}\sin\theta_{2,k}\cdot\cdot\cdot\sin\theta_{d-2,k}\cos\phi_{k}\notag\\
&\cdot\cdot\cdot\\
&X_2(t)=c\sum_{k=1}^{n+1}\tau_k
\sin\theta_{1,k}\cos\theta_{2,k}\notag\\
&X_1(t)=c\sum_{k=1}^{n+1}\tau_k \cos\theta_{1,k}.\notag
\end{align}

The probability distribution of random flight $\underline{\bf X}_d(t),t>0,$ is concentrated inside the $d$-dimensional hypersphere with radius $ct$, which we will indicate by $\mathcal{H}_{ct}^d=\{\underline{\bf x}_d:||\underline{\bf x}_d||<ct\},$ where $\underline{\bf x}_d=(x_1,x_2,...,x_d)$ and $||\cdot||$ represents the Euclidean norm. The density law \eqref{eq:densdir} leads to a drift effect in the random motion. Indeed, since the directions are not chosen with the same probability, the particle tends to move in some portions of the space $\mathbb{R}^d$ with higher probability. Clearly the parameter $\nu$ defines the degree of drift of the motion. Therefore, if the value of $\nu$ increases, the particle will tend to move along the same part of the space. This behavior of the random motion seems to be consistent with the persistence of the real motions. 

The sample paths described by the moving particle having random position \eqref{polar} appear as straight lines with sharp turns and each steps are randomly oriented and with random length (see Figure 2).

The drawbacks emerging in the study of the random flights defined in this paper are discussed in Section 2. Therefore instead of studying directly $\underline{\bf X}_d(t),t>0,$ we analyze its projections onto the lower space $\mathbb{R}^m,m<d$. This approach leads to another random motions for which we obtain the probability distributions. 

For $\nu=1$, we provide the characteristic function for $\underline{\bf X}_d(t),t>0$ (Section 3).  Furthermore, we infer a closed-form expression for the density function of the random flight with $\nu=1$. For $n=1,2,$ it is possible to explicit completely the above density function.

The Appendix contains results on the integrals involving Bessel functions, which assume a crucial role in this paper.

\section{Motions related to the random flights}
In order to study the process $\underline{\bf X}_d(t),t>0,$ introduced in the previous Section, we take into account the same approach developed in Orsingher and De Gregorio (2007) and De Gregorio and Orsingher (2011). The first step consists in the calculation of the characteristic function of $\underline{\bf X}_d(t),t>0,$ which plays a key role in our analysis. Let us indicate by $<\cdot,\cdot>$ the scalar product, then we have that
\begin{align}\label{eq:cfgeneral}
\mathcal{F}_n^\nu(\underline{\alpha}_d)&=E\left\{e^{i<\underline{\alpha}_d,\underline{\bf X}_d(t)>}\right\}\notag\\
&=\int_0^{t}d\tau_1\int_0^{t-\tau_1}d\tau_2\cdots\int_0^{t-\sum_{k=1}^{n-1}\tau_k}d\tau_n\,f_{d,\nu}(\tau_1,...,\tau_n)\,\mathcal{I}_n^\nu(\tau_1,...,\tau_n,\underline{\alpha}_d),
\end{align}
where $\underline\alpha_d=(\alpha_1,\alpha_2,...,\alpha_d)$ and
\begin{align}
&\mathcal{I}_n^\nu(\tau_1,...,\tau_n,\underline{\alpha}_d)\\
&=\int_0^\pi d \theta_{1,1}\cdots\int_0^\pi d \theta_{1,n+1}\cdots \int_0^\pi d \theta_{d-2,1}\cdots\int_0^\pi d \theta_{d-2,n+1} \int_0^{2\pi}d \phi_{1}\cdots\int_0^{2\pi} d \phi_{n+1}\notag\\
&\quad \prod_{k=1}^{n+1}\Bigg\{\exp\{ic\tau_k(\alpha_d\sin\theta_{1,k}\sin\theta_{2,k}\cdot\cdot\cdot\sin\theta_{d-2,k}\sin\phi_{k}+\alpha_{d-1}\sin\theta_{1,k}\sin\theta_{2,k}\cdot\cdot\cdot\sin\theta_{d-2,k}\cos\phi_{k}\notag\\
&\quad+\cdots+\alpha_2\sin\theta_{1,k}\cos\theta_{2,k}+\alpha_1 \cos\theta_{1,k} ) \}\frac{\Gamma(\nu+d/2)}{2\pi^{\frac{d-1}{2}}\Gamma(\nu+1/2)}\sin\theta_{1,k}^{2\nu+d-2}\cdot\cdot\cdot\sin\theta_{d-2,k}^{2\nu+1}\sin\phi_k^{2\nu}\Bigg\}\notag.
\end{align}

The integrals with respect to the $n+1$ angles $\phi_k$ are of the form 
\begin{equation*}\label{eq:intgen}
\int_0^{2\pi}e^{iz(\alpha\cos\phi+\beta\sin\phi)}\sin^{2\nu}\phi d\phi
\end{equation*}
which seems to have an explicit solution only for particular values of $\nu$ (see Appendix). Thus the choice of the non-uniform family of distributions \eqref{eq:densdir} for the directions of the random motion, makes the analysis of the process, at least in its general setting, very complicated. Clearly for $\nu=0$ we get the uniform case studied in the paper mentioned in the Introduction. 

For this reason, instead of studying directly $\underline{\bf X}_{d}(t)$, we deal with the random process $\underline{\bf X}_{m}^d(t)=(X_1(t),...,X_{m}(t)),t>0,$ namely the projection of 
$\underline{\bf X}_{d}(t)$ onto the lower space $\mathbb{R}^{m}$, with $1\leq m<d$, having components equal to
\begin{align}\label{polar2}
&X_m(t)=c\sum_{k=1}^{n+1}\tau_k
\sin\theta_{1,k}\sin\theta_{2,k}\cdot\cdot\cdot\sin\theta_{m-1,k}\cos\theta_{m,k}\notag\\
&X_{m-1}(t)=c\sum_{k=1}^{n+1}\tau_k
\sin\theta_{1,k}\sin\theta_{2,k}\cdot\cdot\cdot\sin\theta_{m-2,k}\cos\theta_{m-1,k}\notag\\
&\cdot\cdot\cdot\\
&X_2(t)=c\sum_{k=1}^{n+1}\tau_k
\sin\theta_{1,k}\cos\theta_{2,k}\notag\\
&X_1(t)=c\sum_{k=1}^{n+1}\tau_k \cos\theta_{1,k}.\notag
\end{align}
The vector $(\theta_1,\theta_2,...,\theta_m)$ (with $\theta_{d-1}=\phi$) appearing in \eqref{polar2}, has distribution given by the marginal density of $g_{d,\nu}(\theta_1,\theta_2,...,\theta_{d-2},\phi)$, that is
\begin{align}\label{eq:densorient2}
g_{d,\nu}(\theta_1,\theta_2,...,\theta_{m})&=\int_0^\pi d\theta_{m+1}\cdots\int_0^\pi d\theta_{d-2}\int_0^{2\pi} d\phi g_{d,\nu}(\theta_1,\theta_2,...,\theta_{d-2},\phi)\notag\\
&=\begin{cases}
g_{d,\nu}(\theta_1,\theta_2,...,\theta_{d-2},\phi), &m=d-1\\\\
\frac{\Gamma(\nu+d/2)}{\pi^{\frac{m}{2}}\Gamma(\nu+\frac{d-m}{2})}\sin\theta_{1}^{2\nu+d-2}\cdot\cdot\cdot\sin\theta_{m}^{2\nu+d-m-1}, & m<d-1
\end{cases}
\end{align}
Therefore, $\underline{\bf X}_{m}^d(t)$ can be interpreted as the shadow of \eqref{polar} in the lower spaces (see for instance Figure 2).
In other words, if we observe in $\mathbb{R}^{m}$ the particle moving in $\mathbb{R}^{d}$ according to the random rules of $\underline{\bf X}_d(t)$, we perceive a random flight with components \eqref{polar2} having vector velocity   
$$\vec{v}=\left(
 \begin{array}{l} 
      c\sin\theta_1\sin\theta_2\cdot\cdot\cdot\sin\theta_{m-1}\cos\theta_m \\
     c\sin\theta_1\sin\theta_2\cdot\cdot\cdot\sin\theta_{m-2}\cos\theta_{m-1} \\
     ...\\
     c\sin\theta_1\cos\theta_2\\
     c\cos\theta_1
   \end{array}    
    \right)$$
with random intensity $|\vec{v}|$. Then, in what follows we analyze the random motion $\underline{\bf X}_{m}^d(t),t>0,$ with orientations distributed according to \eqref{eq:densorient2}. Our first result concerns the characteristic function of \eqref{polar2}.

\begin{theorem} \label{teo1}
The characteristic function of $\underline{\bf X}_{m}^d(t),t>0,$ is equal to 
\begin{align}\label{eq:fc}
\mathcal{F}_{n,d}^\nu(\underline{\alpha}_{m})&=E\left\{e^{i<\underline{\alpha}_m,\underline{\bf X}_m^d(t)>}\right\}\notag\\
&=\frac{2^{\frac{n+1}{2}(2\nu+d-1)-\frac12}\Gamma(\frac{n+1}{2}(2\nu+d-1)+\frac12)}{(ct||\underline{\alpha}_m||)^{\frac{n+1}{2}(2\nu+d-1)-\frac12}}J_{\frac{n+1}{2}(2\nu+d-1)-\frac12}(ct||\underline{\alpha}_m||),
\end{align}
where $J_\mu(x)=\sum_{k=0}^\infty\frac{(-1)^k}{k!\Gamma(k+\mu+1)}\left(\frac{x}{2}\right)^{2k+\mu},\mu\in\mathbb{R},$ is the Bessel function.
\end{theorem}
\begin{proof}
The characteristic function of  $\underline{\bf X}_{m}^d(t),t>0,$ becomes
\begin{align*}
\mathcal{F}_{n,d}^\nu(\underline{\alpha}_m)&=E\left\{e^{i<\underline{\alpha}_m,\underline{\bf X}_m^d(t)>}\right\}\\
&=\int_0^{t}d\tau_1\int_0^{t-\tau_1}d\tau_2\cdots\int_0^{t-\sum_{k=1}^{n-1}\tau_k}d\tau_n\,f_{d,\nu}(\tau_1,...,\tau_n)\,\mathcal{I}_{n,d}^\nu(\tau_1,...,\tau_n,\underline{\alpha}_m)
\end{align*}
with

\begin{align}\label{eq:intpr}
&\mathcal{I}_{n,d}^\nu(\tau_1,...,\tau_n,\underline{\alpha}_{m})\\
&=\int_0^\pi d \theta_{1,1}\cdots\int_0^\pi d \theta_{1,n+1}\cdots \int_0^\pi d \theta_{m,1}\cdots\int_0^\pi d \theta_{m,n+1} \notag\\
&\quad \prod_{k=1}^{n+1}\Bigg\{\exp\{ic\tau_k(\alpha_{m}\sin\theta_{1,k}\sin\theta_{2,k}\cdot\cdot\cdot\sin\theta_{m-1,k}\cos\theta_{m,k}+\cdots+\alpha_2\sin\theta_{1,k}\cos\theta_{2,k}+\alpha_1 \cos\theta_{1,k} ) \}\notag\\
&\quad\times\frac{\Gamma(\nu+d/2)}{\pi^{\frac{m}{2}}\Gamma(\nu+\frac{d-m}{2})}\sin\theta_{1,k}^{2\nu+d-2}\cdot\cdot\cdot\sin\theta_{m,k}^{2\nu+d-m-1}\Bigg\}\notag.
\end{align}
Since
\begin{eqnarray}\label{irbess}
J_\nu(z)&=&\frac{\left(\frac z2\right)^\nu}{\Gamma\left(\nu+\frac12\right)\Gamma\left(\frac12\right)}\int_0^\pi e^{iz\cos\phi}\sin^{2\nu}\phi d\phi,\quad Re\left(\nu+\frac12\right)>0,
\end{eqnarray}
we observe that, after the integrations with respect to $\theta_{m,k}, k=1,...,n+1,$ \eqref{eq:intpr} becomes
\begin{align*}
&\mathcal{I}_{n,d}^\nu(\tau_1,...,\tau_n,\underline{\alpha}_{m})\\
&=\int_0^\pi d \theta_{1,1}\cdots\int_0^\pi d \theta_{1,n+1}\cdots \int_0^\pi d \theta_{m-1,1}\cdots\int_0^\pi d \theta_{m-1,n+1}\notag\\
&\quad \prod_{k=1}^{n+1}\Bigg\{\exp\{ic\tau_k(\alpha_{m-1}\sin\theta_{1,k}\cos\theta_{2,k}\cdot\cdot\cdot\sin\theta_{m-2,k}\cos\theta_{m-1,k}+\cdots+\alpha_2\sin\theta_{1,k}\cos\theta_{2,k}+\alpha_1 \cos\theta_{1,k} ) \}\notag\\
&\quad\sin\theta_{1,k}^{2\nu+d-2}\cdot\cdot\cdot\sin\theta_{m-1,k}^{2\nu+d-m}\Bigg\}\notag\\
&\quad\prod_{k=1}^{n+1}\left\{\frac{\Gamma(\nu+d/2)}{\pi^{\frac{m}{2}}\Gamma(\nu+\frac{d-m}{2})}\int_0^\pi \exp\{ic\tau_k(\alpha_{m}\sin\theta_{1,k}\cos\theta_{2,k}\cdot\cdot\cdot\sin\theta_{m-1,k}\cos\theta_{m,k} ) \}\sin\theta_{m,k}^{2\nu+d-m-1}\right\}\notag\\
&=\int_0^\pi d \theta_{1,1}\cdots\int_0^\pi d \theta_{1,n+1}\cdots \int_0^\pi d \theta_{m-1,1}\cdots\int_0^\pi d \theta_{m-1,n+1}\notag\\
&\quad \prod_{k=1}^{n+1}\Bigg\{\exp\{ic\tau_k(\alpha_{m-1}\sin\theta_{1,k}\cos\theta_{2,k}\cdot\cdot\cdot\sin\theta_{m-2,k}\cos\theta_{m-1,k}+\cdots+\alpha_2\sin\theta_{1,k}\cos\theta_{2,k}+\alpha_1 \cos\theta_{1,k} ) \}\notag\\
&\quad\times\frac{\Gamma(\nu+d/2)}{\pi^{\frac{m-1}{2}}}\sin\theta_{1,k}^{\nu+d-2-\frac{d-m-1}{2}}\cdot\cdot\cdot\sin\theta_{m-1,k}^{\nu+d-m-\frac{d-m-1}{2}}\\
&\quad\times\left(\frac{2}{c\tau_k\alpha_m}\right)^{\nu+\frac{d-m-1}{2}} J_{\nu+\frac{d-m-1}{2}}(c\tau_k\alpha_{m}\sin\theta_{1,k}\sin\theta_{2,k}\cdot\cdot\cdot\sin\theta_{m-1,k})\Bigg\}\notag.
\end{align*}

We are able to perform all the $(m-1)(n+1)$ integrations with
respect to the angles $\theta_{i,k},1\leq i \leq m-1,k=1,...,n+1,$
by applying successively the formula \eqref{formula4} in the Appendix. The integration with respect to
$\theta_{m-1,k},~k=1,...,n+1$ yields
\begin{equation*}\label{eleven}
\begin{split}
&\int_0^\pi d\theta_{m-1,1}\cdot\cdot\cdot \int_0^\pi
d\theta_{m-1,n+1}\prod_{k=1}^{n+1}e^{ic \tau_k\alpha_{m-1}
\sin\theta_{1,k}\cdot\cdot\cdot\sin\theta_{m-2,k}\cos\theta_{m-1,k}}\sin\theta_{m-1,k}^{\nu+\frac{d-m+1}{2}}\\
&J_{\nu+\frac{d-m-1}{2}}(c\tau_k\alpha_{m}\sin\theta_{1,k}\sin\theta_{2,k}\cdot\cdot\cdot\sin\theta_{m-1,k})\frac{1}{\alpha_m^{\nu+\frac{d-m-1}{2}}}
\\
&=\prod_{k=1}^{n+1}\Bigg\{\int_0^\pi e^{ic
\tau_k\alpha_{m-1}
\sin\theta_{1,k}\cdot\cdot\cdot\sin\theta_{m-2,k}\cos\theta_{m-1,k}}\sin\theta_{m-1,k}^{\nu+\frac{d-m+1}{2}}\\
&\quad\times \frac{1}{\alpha_m^{\nu+\frac{d-m-1}{2}}}J_{\nu+\frac{d-m-1}{2}}(c\tau_k\alpha_{m}\sin\theta_{1,k}\sin\theta_{2,k}\cdot\cdot\cdot\sin\theta_{m-1,k})d\theta_{m-1,k}\Bigg\}\\
&=\prod_{k=1}^{n+1}\Bigg\{2\int_0^{\pi/2}\cos (c
\tau_k\alpha_{m-1}
\sin\theta_{1,k}\cdot\cdot\cdot\sin\theta_{m-2,k}\cos\theta_{m-1,k})\\
&\quad\sin\theta^{\nu+\frac{d-m+1}{2}}_{m-1,k}\frac{1}{\alpha_m^{\nu+\frac{d-m-1}{2}}} J_{\nu+\frac{d-m-1}{2}}(c\tau_k\alpha_{m}\sin\theta_{1,k}\sin\theta_{2,k}\cdot\cdot\cdot\sin\theta_{m-1,k})d\theta_{m-1,k}\Bigg\}\\
&=\prod_{k=1}^{n+1}\left\{2\sqrt{\frac{\pi}{2}}\frac{ J_{\nu+\frac{d-m}{2}}\left(c\tau_k\sin\theta_{1,k}\cdot\cdot\cdot\sin\theta_{m-2,k}\sqrt{\alpha_{m}^2
+\alpha_{m-1}^2}\right)}{(c\tau_k\sin\theta_{1,k}\cdot\cdot\cdot\sin\theta_{m-2,k})^\frac12\left(\sqrt{\alpha_{m}^2
+\alpha_{m-1}^2}\right)^{\nu+\frac{d-m}{2}}}\right\}.
\end{split}
\end{equation*}

In the last step we applied formula \eqref{formula4} for
$
a=c\tau_k\alpha_{m}\sin\theta_{1,k}\cdot\cdot\cdot\sin\theta_{m-2,k},$\,$b=c\tau_k\alpha_{m-1}\sin\theta_{1,k}\cdot\cdot\cdot\sin\theta_{m-2,k}
$
and also considered that
\[
\int_0^\pi \sin(\beta\cos x )(\sin x)^{\nu+1}J_\nu(\alpha\sin x)dx=0.
\]

 The integration with respect to
the variables $\theta_{m-2,1},...,\theta_{m-2,n+1}$ follows
similarly by applying again \eqref{formula4} and yields
\begin{equation*}\label{twelve}
\begin{split}
&\int_0^\pi d\theta_{m-2,1}\cdot\cdot\cdot \int_0^\pi
d\theta_{m-2,n+1}\prod_{k=1}^{n+1}e^{ic \tau_k\alpha_{m-2}
\sin\theta_{1,k}\cdot\cdot\sin\theta_{m-3,k}\cos\theta_{m-2,k}}\sin\theta^{\nu+\frac{d-m+3}{2}}_{m-2,k}\\
&\left\{2\sqrt{\frac{\pi}{2}}\frac{ J_{\nu+\frac{d-m}{2}}\left(c\tau_k\sin\theta_{1,k}\cdot\cdot\cdot\sin\theta_{m-2,k}\sqrt{\alpha_{m}^2
+\alpha_{m-1}^2}\right)}{(c\tau_k\sin\theta_{1,k}\cdot\cdot\cdot\sin\theta_{m-2,k})^\frac12\left(\sqrt{\alpha_{m}^2
+\alpha_{m-1}^2}\right)^{\nu+\frac{d-m}{2}}}\right\}\\
&=\prod_{k=1}^{n+1}\left(2\sqrt{\frac{\pi}{2}}\right)^2\frac{ J_{\nu+\frac{d-m+1}{2}}\left(c\tau_k\sin\theta_{1,k}\cdot\cdot\sin\theta_{m-3,k}\sqrt{\alpha_{m}^2
+\alpha_{m-1}^2+\alpha_{m-2}^2}\right)}{c\tau_k\sin\theta_{1,k}\cdot\cdot\sin\theta_{m-3,k}\left(\sqrt{\alpha_{m}^2
+\alpha_{m-1}^2+\alpha_{m-2}^2}\right)^{\nu+\frac{d-m+1}{2}}}.
\end{split}
\end{equation*}

Continuing in the same way we obtain that
\begin{equation}\label{intangle}
\mathcal{I}_{n,d}^\nu(\tau_1,...,\tau_n,\underline{\alpha}_{m})=\left\{2^{\nu+\frac d2-1}\Gamma\left(\nu+\frac d2\right)\right\}^{n+1}\prod_{k=1}^{n+1}\frac{ J_{\nu+\frac d2-1}(c\tau_k||\underline{\alpha}_{m}||)}{(c\tau_k||\underline{\alpha}_{m}||)^{\nu+\frac d2-1}}.
\end{equation}

Then we can write that
\begin{align}\label{eq:intdrift}
\mathcal{F}_{n,d}^\nu(\underline{\alpha}_{m})&=\frac{\Gamma((n+1)(2\nu+d-1))}{(\Gamma(2\nu+d-1))^{n+1}}\frac{\left\{2^{\nu+\frac d2-1}\Gamma\left(\nu+\frac d2\right)\right\}^{n+1}}{t^{(n+1)(2\nu+d-1)-1}}\notag\\
&\quad\int_0^{t}d\tau_1\int_0^{t-\tau_1}d\tau_2\cdots\int_0^{t-\sum_{k=1}^{n-1}\tau_k}d\tau_n\prod_{k=1}^{n+1}\tau_k^{\nu+\frac d2-1}\frac{J_{\nu+\frac d2-1}(c\tau_k||\underline{\alpha}_{m}||)}{(c||\underline{\alpha}_{m}||)^{\nu+\frac d2-1}}.
\end{align}

In order to work out this $n$-fold integral, the result \eqref{formula1} assumes a crucial role. Indeed, we apply recursively the formula \eqref{formula1} to calculate each integral with respect to the variables $\tau_j$. Therefore, in the first step we have that
\begin{align*}
&\int_0^{t-\sum_{k=1}^{n-1}\tau_k}\frac{d\tau_n}{(c||\underline{\alpha}_m||)^{2\nu+d-2}}\tau_{n}^{\nu+\frac d2-1}(t-\sum_{k=1}^{n}\tau_k)^{\nu+\frac d2-1}J_{\nu+\frac d2-1}(c\tau_n||\underline{\alpha}_m||)J_{\nu+\frac d2-1}(c(t-\sum_{k=1}^{n}\tau_k)||\underline{\alpha}_m||)\\
&=\int_0^{t-\sum_{k=1}^{n-1}\tau_k}\frac{d \tau_n}{(c||\underline{\alpha}_m||)^{4\nu+2d-4}}(c\tau_n||\underline{\alpha}_m||)^{\nu+\frac d2-1}(c(t-\sum_{k=1}^{n}\tau_k)||\underline{\alpha}_m||)^{\nu+\frac d2-1}\\
&\quad J_{\nu+\frac d2-1}(c\tau_n||\underline{\alpha}_m||)J_{\nu+\frac d2-1}(c(t-\sum_{k=1}^{n}\tau_k)||\underline{\alpha}_m||)=(c\tau_n||\underline{\alpha}_m||=y )\\
&=\frac{1}{(c||\underline{\alpha}_m||)^{4\nu+2d-3}}\int_0^{c(t-\sum_{k=1}^{n-1}\tau_k)||\underline{\alpha}_m||}dyy^{\nu+\frac d2-1}(c(t-\sum_{k=1}^{n-1}\tau_k)||\underline{\alpha}_m||-y)^{\nu+\frac d2-1}\\
&\quad J_{\nu+\frac d2-1}(y)J_{\nu+\frac d2-1}(c(t-\sum_{k=1}^{n-1}\tau_k)||\underline{\alpha}_d||-y)\\
&=\frac{1}{(c||\underline{\alpha}_m||)^{4\nu+2d-3}}\frac{\left(\Gamma\left(\nu+\frac{d-1}{2}\right)\right)^2}{\sqrt{2\pi}\Gamma(2\nu+d-1)}(c(t-\sum_{k=1}^{n-1}\tau_k)||\underline{\alpha}_m||)^{2\nu+d-\frac32}J_{2\nu+d-\frac 32}(c(t-\sum_{k=1}^{n-1}\tau_k)||\underline{\alpha}_m||).
\end{align*}

The second integral is given by
\begin{eqnarray*}
&&\frac{1}{(c||\underline{\alpha}_m||)^{6\nu+3d-5}}\frac{\left(\Gamma\left(\nu+\frac{d-1}{2}\right)\right)^2}{\sqrt{2\pi}\Gamma(2\nu+d-1)}
\int_0^{t-\sum_{k=1}^{n-2}\tau_k}d\tau_{n-1}\\
&&(c\tau_{n-1}||\underline{\alpha}_m||)^{\nu+\frac d2-1}(c(t-\sum_{k=1}^{n-1}\tau_k)||\underline{\alpha}_m||)^{2\nu+d-\frac32}J_{\nu+\frac d2-1}(c\tau_{n-1}||\underline{\alpha}_m||)J_{2\nu+d-\frac 32}(c(t-\sum_{k=1}^{n-1}\tau_k)||\underline{\alpha}_m||)\\
&&=(c\tau_n||\underline{\alpha}_m||=y )\\
&&=\frac{1}{(c||\underline{\alpha}_m||)^{6\nu+3d-4}}\frac{\left(\Gamma\left(\nu+\frac{d-1}{2}\right)\right)^2}{\sqrt{2\pi}\Gamma(2\nu+d-1)}
\int_0^{c(t-\sum_{k=1}^{n-2}\tau_k)||\underline{\alpha}_m||}dy\\
&&\quad y^{\nu+\frac d2-1}(c(t-\sum_{k=1}^{n-2}\tau_k)||\underline{\alpha}_m||-y)^{2\nu+d-\frac32}J_{\nu+\frac d2-1}(y)J_{2\nu+d-\frac 32}(c(t-\sum_{k=1}^{n-2}\tau_k)||\underline{\alpha}_m||-y)\\
&&=\frac{\left(\Gamma\left(\nu+\frac{d-1}{2}\right)\right)^3}{(\sqrt{2\pi})^2\Gamma(3\nu+\frac32(d-1))}\frac{(c(t-\sum_{k=1}^{n-2}\tau_k)||\underline{\alpha}_m||)^{3\nu+\frac32d-2}}{(c||\underline{\alpha}_m||)^{6\nu+3d-4}}J_{3\nu+\frac32d-2}(c(t-\sum_{k=1}^{n-2}\tau_k)||\underline{\alpha}_m||).
\end{eqnarray*}

Then, the last integral becomes
\begin{align}\label{eq:lastint}
&\frac{\left(\Gamma\left(\nu+\frac{d-1}{2}\right)\right)^n}{(\sqrt{2\pi})^{n-1}\Gamma(n\nu+\frac n2(d-1))}\frac{1}{(c||\underline{\alpha}_m||)^{(n+1)(2\nu+d-1)-2}}
\int_0^{t}d\tau_{1}\\
&(c\tau_{1}||\underline{\alpha}_m||)^{\nu+\frac d2-1}(c(t-\tau_1)||\underline{\alpha}_n||)^{n\nu+\frac n2(d-1)-\frac12}J_{\nu+\frac d2-1}(c\tau_{1}||\underline{\alpha}_m||)J_{n\nu+\frac n2(d-1)-\frac12}(c(t-\tau_1)||\underline{\alpha}_m||)\notag\\
&=(c\tau_1||\underline{\alpha}_m||=y )\notag\\
&=\frac{\left(\Gamma\left(\nu+\frac{d-1}{2}\right)\right)^n}{(\sqrt{2\pi})^{n-1}\Gamma(n\nu+\frac n2(d-1))}\frac{1}{(c||\underline{\alpha}_m||)^{(n+1)(2\nu+d-1)-1}}
\int_0^{ct||\underline{\alpha}_m||}dy\notag\\
&\quad y^{\nu+\frac d2-1}(ct||\underline{\alpha}_m||-y)^{n\nu+\frac n2(d-1)-\frac12}J_{\nu+\frac d2-1}(y)J_{n\nu+\frac n2(d-1)-\frac12}(ct||\underline{\alpha}_m||-y)\notag\\
&=\frac{\left(\Gamma\left(\nu+\frac{d-1}{2}\right)\right)^{n+1}}{(\sqrt{2\pi})^n\Gamma\left((n+1)(\nu+\frac{(d-1)}{2})\right)}\frac{(ct||\underline{\alpha}_m||)^{(n+1)(\nu+\frac{(d-1)}{2})-\frac12}}{(c||\underline{\alpha}_m||)^{(n+1)(2\nu+d-1)-1}}J_{(n+1)(\nu+\frac{(d-1)}{2})-\frac12}(ct||\underline{\alpha}_m||)\notag.
\end{align}
Therefore, plugging the result \eqref{eq:lastint} into the expression \eqref{eq:intdrift}, and by observing that
$$\Gamma\left(\nu+\frac d2\right)=\sqrt{\pi}2^{2-d-2\nu}\frac{\Gamma(2\nu+d-1)}{\Gamma(\nu+\frac{d-1}{2})}$$
and 
$$\Gamma\left(\frac{(n+1)}{2}(2\nu+d-1)+\frac12\right)=\sqrt{\pi}2^{1-(n+1)(2\nu+d-1)}\frac{\Gamma((n+1)(2\nu+d-1))}{\Gamma(\frac{(n+1)}{2}(2\nu+d-1))},$$
some manipulations lead to \eqref{eq:fc}.

\end{proof}


Let us denote by $\underline{\bf x}_m=(x_1,x_2,...,x_m)$ and $d\underline{\bf x}_m=(dx_1,dx_2,...,dx_m)$. Now, by means of \eqref{eq:fc}, we derive the explicit probability distribution of the random motion $\underline{\bf X}_m^d(t),t>0,$ which is concentrated in the hypersphere $\mathcal{H}_{ct}^m$.  

\begin{theorem}\label{th2}
The probability law of $\underline{\bf X}_m^d(t),t>0,$ is equal to
\begin{align}\label{condlaw}
p_{n,d}^\nu(\underline{\bf x}_m,t)&=\frac{P\{\underline{\bf X}_m^d(t)\in d\underline{\bf x}_m \}}{\prod_{i=1}^md x_i}\notag\\
&=\frac{\Gamma(\frac{n+1}{2}(2\nu+d-1)+\frac12)}{\Gamma(\frac{n+1}{2}(2\nu+d-1)-\frac m2+\frac12)}\frac{(c^2t^2- ||\underline{\bf x}_m||^2)^{\frac{n+1}{2}(2\nu+d-1)-\frac {m+1}{2}}
}{\pi^{m/2}(ct)^{(n+1)(2\nu+d-1)-1}},
\end{align}
with $||\underline{\bf x}_m||<ct$ and $d\geq 2$.
\end{theorem}
\begin{proof}
By inverting the characteristic function \eqref{eq:fc}, we are able to show that the 
density law of the process $\underline{\bf X}_m(t),t>0,$ is given by \eqref{condlaw}. Indeed, by passing to the hyperspherical coordinates, we have that
\begin{align*}
p_{n,d}^\nu(\underline{\bf x}_m,t)
&=\frac{1}{(2\pi)^m}\int_{\mathbb{R}^m}e^{-i<\underline{\alpha}_m,\underline{\bf x}_m>}E\left\{e^{i<\underline{\alpha}_m,\underline{\bf X}_m^d(t)>}\right\}d\alpha_1\cdots d\alpha_m\notag\\
&=\frac{1}{(2\pi)^m}\int_0^\infty \rho^{m-1}d\rho\int_0^\pi d\theta_1\cdots\int_0^\pi d\theta_{m-2}\int_0^{2\pi} d	\phi \sin^{m-2}\theta_1\cdots\sin\theta_{m-2}\notag\\
& \exp\{-i\rho(x_m\sin\theta_1\cdots\sin\theta_{m-2}\sin\phi+\cdots+x_2\sin\theta_1\cos\theta_2+x_1\cos\theta_1)\notag\\
&
\notag\frac{2^{\frac{n+1}{2}(2\nu+d-1)-\frac12}\Gamma(\frac{n+1}{2}(2\nu+d-1)+\frac12)}{(ct\rho)^{\frac{n+1}{2}(2\nu+d-1)-\frac12}}J_{\frac{n+1}{2}(2\nu+d-1)-\frac12}(ct\rho),
\notag\\
&=\frac{2^{\frac{n+1}{2}(2\nu+d-1)-\frac12}}{(2\pi)^{m/2}}\Gamma\left(\frac{n+1}{2}(2\nu+d-1)+\frac12\right)\int_0^\infty\rho^{m-1}\frac{J_{\frac m2-1}(\rho ||\underline{\bf x}_m||)}{(\rho||\underline{\bf x}_m||)^{\frac m2-1}}\frac{J_{\frac{n+1}{2}(2\nu+d-1)-\frac12}(ct\rho)}{(ct\rho)^{\frac{n+1}{2}(2\nu+d-1)-\frac12}}d	\rho\notag\\
&=\frac{2^{\frac{n+1}{2}(2\nu+d-1)-\frac12}}{(2\pi)^{m/2}}\frac{\Gamma\left(\frac{n+1}{2}(2\nu+d-1)+\frac12\right)}{(ct)^{\frac{n+1}{2}(2\nu+d-1)-\frac12}||\underline{\bf x}_m||^{\frac m2-1}}\\
&\quad\int_0^\infty \rho^{\frac {m+1}{2}-\frac{n+1}{2}(2\nu+d-1)}J_{\frac m2-1}(\rho ||\underline{\bf x}_m||)J_{\frac{n+1}{2}(2\nu+d-1)-\frac12}(ct\rho)d\rho\notag\\
&=\frac{1}{\pi^{m/2}(ct)^{(n+1)(2\nu+d-1)-1}}\frac{\Gamma(\frac{n+1}{2}(2\nu+d-1)+\frac12)}{\Gamma(\frac{n+1}{2}(2\nu+d-1)-\frac m2+\frac12)}(c^2t^2- ||\underline{\bf x}_m||^2)^{\frac{n+1}{2}(2\nu+d-1)-\frac {m+1}{2}}.
\end{align*}
In the first step above we have performed calculations similar to those leading to \eqref{intangle} and then
\begin{align}\label{intangle2}
&\int_0^\pi d\theta_1\cdots\int_0^\pi d\theta_{m-2}\int_0^{2\pi} d\phi \sin^{m-2}\theta_1\cdots\sin\theta_{m-2}\\
&\exp{\{-i\rho(x_m\sin\theta_1\cdots\sin\theta_{m-2}\sin\phi+\cdots+x_2\sin\theta_1\cos\theta_2+x_1\cos\theta_1)\}}\notag\\
&=(2\pi)^{\frac m2}\frac{J_{\frac m2-1}(\rho ||\underline{\bf x}_m||)}{(\rho ||\underline{\bf x}_m||)^{\frac m2-1}}\notag,
\end{align}
 while in the last step  we have used formula \eqref{formula3}
for $\nu=\frac{n+1}{2}(2\nu+d-1)-\frac32$, $\mu=\frac m2-1$, $a=ct$ and $b= ||\underline{\bf x}_m||$.
\end{proof}
\begin{remark}\label{rem}
By taking into account \eqref{condlaw}, we observe that $\underline{\bf X}_m^d(t),t>0,$ represents an isotropic random walk; that is its probability distribution depends on the distance, from the origin, of the position of the random walker.  Furthermore, by setting $\nu=0$ in \eqref{condlaw}, we reobtain the result (2.26) in De Gregorio and Orsingher (2011).
\end{remark}
\begin{remark}
The cumulative distribution function for the process $\underline{\bf X}_m^d(t),t>0,$ is equal to
\begin{align}\label{eq:cumfun}
P\{\underline{\bf X}_m^d(t)\in\mathcal{H}_r^m\}&=\int_{\mathcal{H}_r^m}p_{n,d}^\nu(\underline{\bf x}_m,t)\prod_{i=1}^mdx_i\notag\\
&=\frac{2\Gamma(\frac{n+1}{2}(2\nu+d-1)+\frac12)}{\Gamma(\frac{n+1}{2}(2\nu+d-1)-\frac m2+\frac12)\Gamma(\frac m2)}\int_0^r\rho^{m-1}\frac{(c^2t^2-\rho^2)^{\frac{n+1}{2}(2\nu+d-1)-\frac {m+1}{2}}
}{(ct)^{(n+1)(2\nu+d-1)-1}}d\rho
\end{align} 
with $0<r<ct$.
If $\frac{n+1}{2}(2\nu+d-1)-\frac {m+1}{2}=q\in \mathbb{N}$, we can write down $P\{\underline{\bf X}_m^d(t)\in\mathcal{H}_r^m\}$ in an alternative form involving a finite sum. Then, one has that
\begin{align*}
P\{\underline{\bf X}_m^d(t)\in\mathcal{H}_r^m\}&=\frac{2\Gamma(q+\frac m2+1)}{\Gamma(q+1)\Gamma(\frac m2)}\int_0^r\rho^{m-1}\frac{(c^2t^2-\rho^2)^{q}
}{(ct)^{2q+m}}d\rho\\
&=\frac{\Gamma(q+\frac m2+1)}{\Gamma(q+1)\Gamma(\frac m2)}\int_0^{(r/ct)^2}y^{\frac m2-1}(1-y)^{q}
dy\\
&=\frac{\Gamma(q+\frac m2+1)}{\Gamma(q+1)\Gamma(\frac m2)}\sum_{k=0}^q(-1)^k\binom{q}{k}\left(\frac{r}{ct}\right)^{2k+ m}\frac{1}{k+\frac m2}
\end{align*} 
\end{remark}
\begin{remark}
By assuming that the number of steps $n+1$ is a random value, it is possible to get the unconditional distribution for $\underline{\bf X}_m^d(t),t>0$. In other words, we suppose that the number $n$ of changes of direction are governed by a fractional Poisson process $N_d^\nu(t),t>0,$ introduced by Orsingher and Beghin (2009). Therefore, we have that 
\begin{equation}\label{eq:fracpoiss}
P\{N_d^\nu(t)=n\}=\frac{1}{\Gamma(\frac{n+1}{2}(2\nu+d-1)+\frac12)}\frac{(\lambda t)^n}{E_{\nu+\frac{d-1}{2},\nu+\frac d2}(\lambda t)},\quad n=0,1,2,...
\end{equation}
where $E_{\alpha,\beta}(x)=\sum_{k=0}\frac{x^k}{k!\Gamma(\alpha k+\beta)}, x\in\mathbb{R},\alpha,\beta>0$ is the Mittag-Leffler function. Furthermore $N_d^\nu(t),t>0,$ is supposed independent from $\underline\tau$ and $\underline\theta$. Then, in order to obtain the unconditional distribution of $\underline{\bf X}_m^d(t),t>0,$ we can average \eqref{condlaw} with the distribution \eqref{eq:fracpoiss}.
\end{remark}

\begin{remark}
Let us consider the radial process $R_m^d(t)=||\underline{\bf X}_m^d(t)||=\sqrt{\sum_{i=1}^mX_i(t)},t>0$. We observe that
\begin{align*}
P\{R_m^d(t)<r\}= P\{\underline{\bf X}_{m}^d(t)\in \mathcal{H}_r^{m}\},
\end{align*}
with $0<r<ct$.
Therefore, let $f_{n,m,d}^{\nu}(r,t)$ be the density function of $R_m^d(t)$, the result \eqref{eq:cumfun}implies that
\begin{align}\label{densradial}
f_{n,m,d}^{\nu}(r,t)=\frac{2\Gamma(\frac{n+1}{2}(2\nu+d-1)+\frac12)}{\Gamma(\frac{n+1}{2}(2\nu+d-1)-\frac m2+\frac12)\Gamma(\frac{m}{2})}\frac{r^{m-1}(c^2t^2-r^2)^{\frac{n+1}{2}(2\nu+d-1)-\frac {m+1}{2}}
}{(ct)^{(n+1)(2\nu+d-1)-1}}.
\end{align}
 Furthermore, by taking into account \eqref{densradial}, we derive the moments of $R_m^d(t),$ that is 
 $$E\{R_m^d(t)\}^p=\frac{\Gamma(\frac{n+1}{2}(2\nu+d-1)+\frac12)\Gamma(\frac{p+m}{2})}{\Gamma(\frac{n+1}{2}(2\nu+d-1)+\frac{p+1}{2})\Gamma(\frac{m}{2})}(ct)^p$$
  with $p\geq 1$.
 \end{remark}

\section{Random flights with $\nu=1$}

As observed in the previous Section for a general value of $\nu$, we are not able to work out the integral \eqref{eq:cfgeneral}. Nevertheless, it is possible to derive a closed-form expression for the characteristic function of $\underline{\bf X}_{d}(t),t>0,$ for some values of $\nu$. The next Theorem provides the explicit characteristic function for the random flight with $\nu=1$. It is worth to mention that this result is remarkable since in this context it is hard to obtain explicit results.

\begin{theorem}\label{teo:cfnu1} 
For $\nu=1$, the process $\underline{\bf X}_{d}(t),t>0,$ admits the following characteristic function
\begin{equation} \label{eq:cfnu1}
\mathcal{F}_n^1(\underline{\alpha}_d)=\frac{\sqrt{\pi}\Gamma((n+1)(d+1))}{2^{\frac{(d+1)(n+1)-1}{2}}}\sum_{j=0}^{n+1}(-1)^{n+1-j}\binom{n+1}{j}\frac{\left(\frac{\alpha_{d}^2}{||\underline{\alpha}_d||^2}\frac{(d+1)}{2}\right)^{n+1-j}}{\Gamma(\frac{(n+1)(d+3)}{2}-j)}\frac{J_{\frac{(n+1)(d+3)-(2j+1)}{2}}(ct||\underline{\alpha}_d||)}{(ct||\underline{\alpha}_d||)^{\frac{(n+1)(d-1)+2j-1}{2}}}
\end{equation}
with $n\geq 1$ and $d\geq 2$.
\end{theorem}
\begin{proof}
In order to prove the result \eqref{eq:cfnu1}, we will use the same approach and tools exploited in the proof of Theorem \ref{teo1}. Nevertheless the proof of \eqref{eq:cfnu1} is not a simple adjustment of the proof of Theorem \ref{teo1} and it requires a particular care as we will show in the next steps.  

We start by taking into account the expression \eqref{eq:cfgeneral} for $\nu=1$. Thus, we have to handle the following integral 
\begin{align}\label{eq:I_1}
&\mathcal{I}_n^1(\tau_1,...,\tau_n,\underline{\alpha}_d)\\
&=\left(\frac{\Gamma(1+d/2)}{\pi^{\frac{d}{2}}}\right)^{n+1}\int_0^\pi d \theta_{1,1}\cdots\int_0^\pi d \theta_{1,n+1}\cdots \int_0^\pi d \theta_{d-2,1}\cdots\int_0^\pi d \theta_{d-2,n+1} \int_0^{2\pi}d \phi_{1}\cdots\int_0^{2\pi} d \phi_{n+1}\notag\\
&\quad \prod_{k=1}^{n+1}\Bigg\{\exp\{ic\tau_k(\alpha_d\sin\theta_{1,k}\sin\theta_{2,k}\cdot\cdot\cdot\sin\theta_{d-2,k}\sin\phi_{k}+\alpha_{d-1}\sin\theta_{1,k}\sin\theta_{2,j}\cdot\cdot\cdot\sin\theta_{d-2,k}\cos\phi_{k}\notag\\
&\quad+\cdots+\alpha_2\sin\theta_{1,k}\cos\theta_{2,k}+\alpha_1 \cos\theta_{1,k} ) \}\sin\theta_{1,k}^{d}\cdot\cdot\cdot\sin^3\theta_{d-2,k}\sin^{2}\phi_k\Bigg\}\notag.
\end{align}
For the $n+1$ integrals with respect to $\phi_k$, we get that
\begin{align*}
&\int_0^{2\pi}d \phi_{1}\cdots\int_0^{2\pi} d \phi_{n+1} \prod_{k=1}^{n+1}e^{ic\tau_k\sin\theta_{1,k}\sin\theta_{2,k}\cdot\cdot\cdot\sin\theta_{d-2,k}(\alpha_d\sin\phi_{k}+\alpha_{d-1}\cos\phi_{k})}\sin\phi_k^{2}d\phi_k\\
&=\prod_{k=1}^{n+1}\left\{\int_0^{2\pi}e^{ic\tau_k\sin\theta_{1,k}\sin\theta_{2,k}\cdot\cdot\cdot\sin\theta_{d-2,k}(\alpha_d\sin\phi_{k}+\alpha_{d-1}\cos\phi_{k})}\sin\phi_k^{2}d\phi_k\right\}\\
&=(2\pi)^{n+1}\prod_{k=1}^{n+1}\Bigg\{\frac{J_1(c\tau_k\sin\theta_{1,k}\sin\theta_{2,k}\cdot\cdot\cdot\sin\theta_{d-2,k}\sqrt{\alpha_d^2+\alpha_{d-1}^2})}{c\tau_k\sin\theta_{1,k}\sin\theta_{2,k}\cdot\cdot\cdot\sin\theta_{d-2,k}\sqrt{\alpha_d^2+\alpha_{d-1}^2}}\\
&\quad-\frac{\alpha_{d}^2}{\alpha_{d}^2+\alpha_{d-1}^2}J_2(c\tau_k\sin\theta_{1,k}\sin\theta_{2,k}\cdot\cdot\cdot\sin\theta_{d-2,k}\sqrt{\alpha_d^2+\alpha_{d-1}^2})\Bigg\},
\end{align*}
where in the last step we have used the result \eqref{eq:int}. Therefore the integral \eqref{eq:I_1} becomes
\begin{align*}
&\mathcal{I}_n^1(\tau_1,...,\tau_n,\underline{\alpha}_d)\\
&=\left(\frac{\Gamma(1+d/2)}{\pi^{\frac{d}{2}}}\right)^{n+1}(2\pi)^{n+1}\int_0^\pi d \theta_{1,1}\cdots\int_0^\pi d \theta_{1,n+1}\cdots \int_0^\pi d \theta_{d-2,1}\cdots\int_0^\pi d \theta_{d-2,n+1}\notag\\
&\quad \prod_{k=1}^{n+1}\Bigg\{e^{ic\tau_k(\alpha_{d-2}\sin\theta_{1,k}\sin\theta_{2,k}\cdot\cdot\cdot\cos\theta_{d-2,k}+\cdots+\alpha_2\sin\theta_{1,k}\cos\theta_{2,k}+\alpha_1 \cos\theta_{1,k} )}\sin\theta_{1,k}^{d}\cdot\cdot\cdot\sin^3\theta_{d-2,k}\\
&\quad\Bigg[\frac{J_1(c\tau_k\sin\theta_{1,k}\sin\theta_{2,k}\cdot\cdot\cdot\sin\theta_{d-2,k}\sqrt{\alpha_d^2+\alpha_{d-1}^2})}{c\tau_k\sin\theta_{1,k}\sin\theta_{2,k}\cdot\cdot\cdot\sin\theta_{d-2,k}\sqrt{\alpha_d^2+\alpha_{d-1}^2}}\\
&\quad-\frac{\alpha_{d}^2}{\alpha_{d}^2+\alpha_{d-1}^2}J_2(c\tau_k\sin\theta_{1,k}\sin\theta_{2,k}\cdot\cdot\cdot\sin\theta_{d-2,k}\sqrt{\alpha_d^2+\alpha_{d-1}^2})\Bigg]\Bigg\}.
\end{align*}

The integrals with respect to $\theta_{d-2,k}$ are equal to
\begin{align*}
&\int_0^\pi d \theta_{d-2,1}\cdots\int_0^\pi d \theta_{d-2,n+1}\prod_{k=1}^{n+1}\Bigg\{e^{ic\tau_k(\alpha_{d-2}\sin\theta_{1,k}\sin\theta_{2,k}\cdot\cdot\cdot\cos\theta_{d-2,k})}\sin^3\theta_{d-2,k}\\
&\Bigg[\frac{J_1(c\tau_k\sin\theta_{1,k}\sin\theta_{2,k}\cdot\cdot\cdot\sin\theta_{d-2,k}\sqrt{\alpha_d^2+\alpha_{d-1}^2})}{c\tau_k\sin\theta_{1,k}\sin\theta_{2,k}\cdot\cdot\cdot\sin\theta_{d-2,k}\sqrt{\alpha_d^2+\alpha_{d-1}^2}}\\
&-\frac{\alpha_{d}^2}{\alpha_{d}^2+\alpha_{d-1}^2}J_2(c\tau_k\sin\theta_{1,k}\sin\theta_{2,k}\cdot\cdot\cdot\sin\theta_{d-2,k}\sqrt{\alpha_d^2+\alpha_{d-1}^2})\Bigg]\Bigg\}\\
&=\prod_{k=1}^{n+1}\Bigg\{2\int_0^{\pi/2} d\theta_{d-2,k}\cos(c\tau_k\alpha_{d-2}\sin\theta_{1,k}\sin\theta_{2,k}\cdot\cdot\cdot\cos\theta_{d-2,k})\sin^3\theta_{d-2,k}\\
&\quad\Bigg[\frac{J_1(c\tau_k\sin\theta_{1,k}\sin\theta_{2,k}\cdot\cdot\cdot\sin\theta_{d-2,k}\sqrt{\alpha_d^2+\alpha_{d-1}^2})}{c\tau_k\sin\theta_{1,k}\sin\theta_{2,k}\cdot\cdot\cdot\sin\theta_{d-2,k}\sqrt{\alpha_d^2+\alpha_{d-1}^2}}\\
&\quad-\frac{\alpha_{d}^2}{\alpha_{d}^2+\alpha_{d-1}^2}J_2(c\tau_k\sin\theta_{1,k}\sin\theta_{2,k}\cdot\cdot\cdot\sin\theta_{d-2,k}\sqrt{\alpha_d^2+\alpha_{d-1}^2})\Bigg]\Bigg\}\\
&=\prod_{k=1}^{n+1}\Bigg\{2\sqrt{\frac \pi2}\Bigg[\frac{J_{3/2}(c\tau_k\sin\theta_{1,k}\sin\theta_{2,k}\cdot\cdot\cdot\sin\theta_{d-3,k}\sqrt{\alpha_d^2+\alpha_{d-1}^2+\alpha_{d-2}^2})}{(c\tau_k\sin\theta_{1,k}\sin\theta_{2,k}\cdot\cdot\cdot\sin\theta_{d-3,k}\sqrt{\alpha_d^2+\alpha_{d-1}^2+\alpha_{d-2}^2})^{3/2}}\\
&\quad-\frac{\alpha_{d}^2J_{5/2}(c\tau_k\sin\theta_{1,k}\sin\theta_{2,k}\cdot\cdot\cdot\sin\theta_{d-3,k}\sqrt{\alpha_d^2+\alpha_{d-1}^2+\alpha_{d-2}^2})}{(c\tau_k\sin\theta_{1,k}\sin\theta_{2,k}\cdot\cdot\cdot\sin\theta_{d-3,k})^{1/2}(\sqrt{\alpha_d^2+\alpha_{d-1}^2+\alpha_{d-2}^2})^{5/2}}\Bigg]\Bigg\},
\end{align*}
where in the last step we have used formula \eqref{formula4} for 
$$a=c\tau_k\sin\theta_{1,k}\sin\theta_{2,k}\cdot\cdot\cdot\sin\theta_{d-3,k}\sqrt{\alpha_d^2+\alpha_{d-1}^2}$$
$$b=c\tau_k\alpha_{d-2}\sin\theta_{1,k}\sin\theta_{2,k}\cdot\cdot\cdot\sin\theta_{d-3,k}$$
and $\nu=1,2$ and also considered that
$$
\int_0^{\pi}(\sin x)^{\nu+1}\sin(b\cos x)J_\nu(a\sin
x)dx=0$$

By means of the same arguments, we can calculate the further integrals with respect to $\theta_{1,k},...,\theta_{d-4,k},,\theta_{d-3,k}$ and then we obtain that
\begin{equation}
\mathcal{I}_n^1(\tau_1,...,\tau_n,\underline{\alpha}_d)=\left\{2^{d/2}\Gamma(1+d/2)\right\}^{n+1}\prod_{k=1}^{n+1}\left\{\frac{J_{d/2}(c\tau_k||\underline{\alpha}_d||)}{(c\tau_k||\underline{\alpha}_d||)^{d/2}}-\frac{\alpha_{d}^2J_{d/2+1}(c\tau_k||\underline{\alpha}_d||)}{(c\tau_k)^{d/2-1}||\underline{\alpha}_d||^{d/2+1}}\right\}
\end{equation}
 
 The characteristic function for $\nu=1$ becomes
 \begin{align}\label{eq:cfintegralsum}
 \mathcal{F}_n^1(\underline{\alpha}_d)
 &=\left\{2^{d/2}\Gamma(1+d/2)\right\}^{n+1}\int_0^{t}d\tau_1\int_0^{t-\tau_1}d\tau_2\cdots\int_0^{t-\sum_{k=1}^{n-1}\tau_k}d\tau_n\,f_{d,1}(\tau_1,...,\tau_n)\notag\\
 &\quad\prod_{k=1}^{n+1}\left\{\frac{J_{d/2}(c\tau_k||\underline{\alpha}_d||)}{(c\tau_k||\underline{\alpha}_d||)^{d/2}}-\frac{\alpha_{d}^2J_{d/2+1}(c\tau_k||\underline{\alpha}_d||)}{(c\tau_k)^{d/2-1}||\underline{\alpha}_d||^{d/2+1}}\right\}
\end{align}

By developing the product appearing in \eqref{eq:cfintegralsum}, the characteristic function $ \mathcal{F}_n^1(\underline{\alpha}_d)$ reduces to a sum, over to $2^{n+1}$ elements, with generic terms given by 
  \begin{align}\label{eq:mixintgen}
  &\mathcal{J}_{n,j}^{n_1,n_2,...,n_s}(\underline{\alpha}_d)\\
 &=\left\{\frac{2^{d/2}\Gamma(1+d/2)}{\Gamma(d+1)}\right\}^{n+1}\frac{\Gamma((n+1)(d+1))}{t^{(n+1)(d+1)-1}}\left(-\frac{\alpha_{d}^2}{||\underline{\alpha}_d||^2}\right)^{n+1-j}\notag\\
 &\quad\int_0^{t}\tau_1^{d}d\tau_1\int_0^{t-\tau_1}\tau_2^dd\tau_2\cdots\int_0^{t-\sum_{k=1}^{n_1-2}}\tau_{n_1-1}^dd\tau_{n_1-1}\int_0^{t-\sum_{k=1}^{n_1-1}\tau_k}\tau_{n_1}^dd\tau_{n_1}\prod_{k=1}^{n_1}\frac{J_{d/2}(c\tau_k||\underline{\alpha}_d||)}{(c\tau_k||\underline{\alpha}_d||)^{d/2}}\notag\\
 &\quad\int_0^{t-\sum_{k=1}^{n_1}\tau_{n_1+1}}\tau_{n_1+1}^{d}d\tau_{n_1+1}\int_0^{t-\sum_{k=1}^{n_1+1}\tau_{n}}\tau_{n_1+2}^dd\tau_{n_1+2}\cdots\notag\\
 &\quad\int_0^{t-\sum_{k=1}^{n_2-2}\tau_k}\tau_{n_2-1}^dd\tau_{n_2-1}\int_0^{t-\sum_{k=1}^{n_2-1}\tau_k}\tau_{n_2}^dd\tau_{n_2}\prod_{k=n_1+1}^{n_2}\frac{J_{d/2+1}(c\tau_k||\underline{\alpha}_d||)}{(c\tau_k||\underline{\alpha}_d||)^{d/2-1}}\notag\\
 &\quad\cdots\notag\\
  &\quad\int_0^{t-\sum_{k=1}^{n_{s-1}}\tau_k}\tau_{n_{s-1}+1}^{d}d\tau_{n_{s-1}+1}\int_0^{t-\sum_{k=1}^{n_{s-1}+1}\tau_k}\tau_{n_{s-1}+2}^dd\tau_{n_{s-1}+2}\cdots\notag\\
  &\quad\int_0^{t-\sum_{k=1}^{n_s-2}\tau_k}\tau_{n_{s}-1}^dd\tau_{n_{s}-1}\int_0^{t-\sum_{k=1}^{n_s-1}\tau_k}\tau_{n_{s}}^dd\tau_{n_{s}}\prod_{k=n_{s-1}+1}^{n_{s}}\frac{J_{d/2}(c\tau_k||\underline{\alpha}_d||)}{(c\tau_k||\underline{\alpha}_d||)^{d/2}}\notag\\
 &\quad\int_0^{t-\sum_{k=1}^{n_{s}}\tau_k}\tau_{n_s+1}^{d}d\tau_{n_s+1}\int_0^{t-\sum_{k=1}^{n_{s}+1}\tau_k}\tau_{{n_{s}+2}}^dd\tau_{{n_{s}+2}}\cdots\notag\\
 &\quad\int_0^{t-\sum_{k=1}^{n-2}\tau_k}\tau_{n-1}^dd\tau_{n-1}\int_0^{t-\sum_{k=1}^{n-1}\tau_k}\tau_{n}^d(t-\sum_{k=1}^{n}\tau_k)^dd\tau_{n}\prod_{k=n_{s}+1}^{n+1}\frac{J_{d/2+1}(c\tau_k||\underline{\alpha}_d||)}{(c\tau_k||\underline{\alpha}_d||)^{d/2-1}}\notag,
   \end{align}
    with $0\leq n_1<n_2<...<n_{s-1}<n_s	\leq n+1$, $j=n_1+n_3-n_2+...+n_{s}-n_{s-1}$ and $0\leq j\leq n+1$. Therefore, we focus our attention on the calculation of \eqref{eq:mixintgen}. 
   
We deal with the following $(n-n_s)$-fold integral
 \begin{align*}
\mathcal{K}^{n_s}(\underline{\alpha}_d)=&\int_0^{t-\sum_{k=1}^{n_{s}}\tau_k}\tau_{n_s+1}^{d}d\tau_{n_s+1}\int_0^{t-\sum_{k=1}^{n_{s}+1}\tau_k}\tau_{{n_{s}+2}}^dd\tau_{{n_{s}+2}}\cdots\notag\\
 &\int_0^{t-\sum_{k=1}^{n-2}\tau_k}\tau_{n-1}^dd\tau_{n-1}\int_0^{t-\sum_{k=1}^{n-1}\tau_k}\tau_{n}^d(t-\sum_{k=1}^{n}\tau_k)^dd\tau_{n}\prod_{k=n_{s}+1}^{n+1}\frac{J_{d/2+1}(c\tau_k||\underline{\alpha}_d||)}{(c\tau_k||\underline{\alpha}_d||)^{d/2-1}}
\end{align*}
   by applying formula \eqref{formula1}. Indeed, the integral with respect to $\tau_n$ is given by
\begin{align*}
&\int_0^{t-\sum_{k=1}^{n-1}\tau_k}\tau_n^d(t-\sum_{k=1}^n\tau_k)^d\frac{J_{d/2+1}(c\tau_n||\underline{\alpha}_d||)}{(c\tau_n||\underline{\alpha}_d||)^{d/2-1}}\frac{J_{d/2+1}(c(t-\sum_{k=1}^{n}\tau_k)||\underline{\alpha}_d||)}{(c(t-\sum_{k=1}^{n}\tau_k)||\underline{\alpha}_d||)^{d/2-1}}d\tau_n\\
&=\int_0^{t-\sum_{k=1}^{n-1}\tau_k}\frac{(c\tau_n||\underline{\alpha}_d||)^{d/2+1}(c(t-\sum_{k=1}^{n}\tau_k)||\underline{\alpha}_d||)^{d/2+1}}{(c||\underline{\alpha}_d||)^{2d}}J_{d/2+1}(c\tau_n||\underline{\alpha}_d||)J_{d/2+1}(c(t-\sum_{k=1}^{n}\tau_k)||\underline{\alpha}_d||)d\tau_n\\
&=(c\tau_n||\underline{\alpha}_d||=w)\\
&=\int_0^{c(t-\sum_{k=1}^{n-1}\tau_k)||\underline{\alpha}_d||}\frac{w^{d/2+1}(c(t-\sum_{k=1}^{n-1}\tau_k)||\underline{\alpha}_d||-w)^{d/2+1}}{(c||\underline{\alpha}_d||)^{2d+1}}J_{d/2+1}(w)J_{d/2+1}(c(t-\sum_{k=1}^{n-1}\tau_k)||\underline{\alpha}_d||-w)dw\\
&=\frac{1}{(c||\underline{\alpha}_d||)^{2d+1}}\frac{\Gamma^2((d+3)/2)}{\sqrt{2\pi}\Gamma(d+3)}(c(t-\sum_{k=1}^{n-1}\tau_k)||\underline{\alpha}_d||)^{d+5/2}J_{d+5/2}(c(t-\sum_{k=1}^{n-1}\tau_k)||\underline{\alpha}_d||)
\end{align*} 
where in the last step we have used formula \eqref{formula1}. The integral with respect to $\tau_{n-1}$ becomes
\begin{align*}
&\frac{1}{(c||\underline{\alpha}_d||)^{3d+1}}\frac{\Gamma^2((d+3)/2)}{\sqrt{2\pi}\Gamma(d+3)}\\
&\quad\int_0^{t-\sum_{k=1}^{n-2}\tau_k}(c\tau_{n-1}||\underline{\alpha}_d||)^{d/2+1}(c(t-\sum_{k=1}^{n-1}\tau_k)||\underline{\alpha}_d||)^{d+5/2}J_{d/2+1}(c\tau_{n-1}||\underline{\alpha}_d||)J_{d+5/2}(c(t-\sum_{k=1}^{n-1}\tau_k)||\underline{\alpha}_d||)d\tau_{n-1}\\
&=(c\tau_{n-1}||\underline{\alpha}_d||=w)\\
&=\frac{1}{(c||\underline{\alpha}_d||)^{3d+2}}\frac{\Gamma^2((d+3)/2)}{\sqrt{2\pi}\Gamma(d+3)}\\
&\quad\int_0^{c(t-\sum_{k=1}^{n-2}\tau_k)\underline{\alpha}_d||}w^{d/2+1}(c(t-\sum_{k=1}^{n-2}\tau_k)||\underline{\alpha}_d||-w)^{d+5/2}J_{d/2+1}(w)J_{d+5/2}(c(t-\sum_{k=1}^{n-2}\tau_k)||\underline{\alpha}_d||-w)dw\\
&=\frac{1}{(c||\underline{\alpha}_d||)^{3d+2}}\frac{\Gamma^3((d+3)/2)}{(\sqrt{2\pi})^2\Gamma(\frac32(d+3))}(c(t-\sum_{k=1}^{n-2}\tau_k)||\underline{\alpha}_d||)^{\frac32d+4}J_{\frac32d+4}(c(t-\sum_{k=1}^{n-2}\tau_k)||\underline{\alpha}_d||),
\end{align*} 
where in the last step we have used the formula \eqref{formula1} again. We can continue at the same way and then we have that
\begin{align*}
\mathcal{K}^{n_s}(\underline{\alpha}_d)=&\frac{1}{(c||\underline{\alpha}_d||)^{(d+1)(n-n_s+1)-2}}\frac{\Gamma^{n-n_s}((d+3)/2)}{(\sqrt{2\pi})^{n-n_s-1}\Gamma(\frac{n-n_s}{2}(d+3))}\\
&\int_0^{t-\sum_{k=1}^{n_s}\tau_k}d\tau_{n_s+1}(c\tau_{n_s+1}||\underline{\alpha}_d||)^{d/2+1}(c(t-\sum_{k=1}^{n_s+1}\tau_k)||\underline{\alpha}_d||)^{\frac{(n-n_s)(d+3)}{2}-\frac12}\\
&\quad J_{d/2+1}(c\tau_{n_s+1}||\underline{\alpha}_d||)J_{\frac{(n-n_s)(d+3)}{2}-\frac12}(c(t-\sum_{k=1}^{n_s+1}\tau_k)||\underline{\alpha}_d||)\\
&=(c\tau_{n_s+1}||\underline{\alpha}_d||=w)\\
&=\frac{1}{(c||\underline{\alpha}_d||)^{(d+1)(n-n_s+1)-1}}\frac{\Gamma^{n-n_s}((d+3)/2)}{(\sqrt{2\pi})^{n-n_s-1}\Gamma(\frac{n-n_s}{2}(d+3))}\\
&\int_0^{c(t-\sum_{k=1}^{n_s}\tau_k)||\underline{\alpha}_d||}dww^{d/2+1}(c(t-\sum_{k=1}^{n_s}\tau_k)||\underline{\alpha}_d||-w)^{\frac{(n-n_s)(d+3)}{2}-\frac12}\\
&\quad J_{d/2+1}(w)J_{\frac{(n-n_s)(d+3)}{2}-\frac12}(c(t-\sum_{k=1}^{n_s}\tau_k)||\underline{\alpha}_d||-w)\\
&=\frac{\Gamma^{n+1-n_s}((d+3)/2)(c(t-\sum_{k=1}^{n_s}\tau_k)||\underline{\alpha}_d||)^{\frac{(n+1-n_s)(d+3)}{2}-\frac12}}{(c||\underline{\alpha}_d||)^{(d+1)(n-n_s+1)-1}(\sqrt{2\pi})^{n-n_s}\Gamma(\frac{n+1-n_s}{2}(d+3))}J_{\frac{(n+1-n_s)(d+3)}{2}-\frac12}(c(t-\sum_{k=1}^{n_s}\tau_k)||\underline{\alpha}_d||).
\end{align*}

Now, we perform the calculations concerning the $(n_s-n_{s-1})$-fold integral 
\begin{align*}
&\mathcal{K}^{n_{s-1},n_s}(\underline{\alpha}_d)\\
&=\int_0^{t-\sum_{k=1}^{n_{s-1}}\tau_k}\tau_{n_{s-1}+1}^{d}d\tau_{n_{s-1}+1}\int_0^{t-\sum_{k=1}^{n_{s-1}+1}\tau_k}\tau_{n_{s-1}+2}^dd\tau_{n_{s-1}+2}\cdots\notag\\
  &\quad\int_0^{t-\sum_{k=1}^{n_s-2}\tau_k}\tau_{n_{s}-1}^dd\tau_{n_{s}-1}\int_0^{t-\sum_{k=1}^{n_s-1}\tau_k}\tau_{n_{s}}^dd\tau_{n_{s}}\mathcal{K}^{n_s}(\underline{\alpha}_d)\prod_{k=n_{s-1}+1}^{n_{s}}\frac{J_{d/2}(c\tau_k||\underline{\alpha}_d||)}{(c\tau_k||\underline{\alpha}_d||)^{d/2}}.
\end{align*}

By means of the same approach adopted above, we can write down that
\begin{align*}
&\frac{\Gamma^{n+1-n_s}((d+3)/2)
}{(c||\underline{\alpha}_d||)^{(d+1)(n-n_s+1)-1+d}(\sqrt{2\pi})^{n-n_s}\Gamma(\frac{n+1-n_s}{2}(d+3))}\\
&\int_0^{t-\sum_{k=1}^{n_s-1}\tau_k}d\tau_{n_s}(c\tau_{n_s}||\underline{\alpha}_d||)^{d/2}(c(t-\sum_{k=1}^{n_s}\tau_k)||\underline{\alpha}_d||)^{\frac{(n+1-n_s)(d+3)}{2}-\frac12}\\
&\quad J_{d/2}(c\tau_{n_s}||\underline{\alpha}_d||)J_{\frac{(n+1-n_s)(d+3)}{2}-\frac12}(c(t-\sum_{k=1}^{n_s}\tau_k)||\underline{\alpha}_d||)\\
&=(c\tau_{n_s}||\underline{\alpha}_d||=w)\\
&=\frac{\Gamma^{n+1-n_s}((d+3)/2)
}{(c||\underline{\alpha}_d||)^{(d+1)(n-n_s+1)+d}(\sqrt{2\pi})^{n-n_s}\Gamma(\frac{n+1-n_s}{2}(d+3))}\\
&\quad\int_0^{c(t-\sum_{k=1}^{n_s-1}\tau_k)||\underline{\alpha}_d||}dww^{d/2}(c(t-\sum_{k=1}^{n_s-1}\tau_k)||\underline{\alpha}_d||-w)^{\frac{(n+1-n_s)(d+3)}{2}-\frac12}\\
&\qquad J_{d/2}(w)J_{\frac{(n+1-n_s)(d+3)}{2}-\frac12}(c(t-\sum_{k=1}^{n_s-1}\tau_k)||\underline{\alpha}_d||-w)\\
&=\frac{\Gamma^{n+1-n_s}((d+3)/2)\Gamma((d+1)/2)}{(c||\underline{\alpha}_d||)^{(d+1)(n-n_s+1)+d}(\sqrt{2\pi})^{n+1-n_s}\Gamma(\frac{(n+1-(n_s-1))(d+3)}{2}-1)}\\
&\quad(c(t-\sum_{k=1}^{n_s-1}\tau_k)||\underline{\alpha}_d||)^{\frac{(n+1-(n_s-1))(d+3)}{2}-\frac32}J_{\frac{(n+1-(n_s-1))(d+3)}{2}-\frac32}(c(t-\sum_{k=1}^{n_s-1}\tau_k)||\underline{\alpha}_d||).
\end{align*} 

 Then, carrying on the same calculations for the integrals with respect to the variables $\tau_{n_s-1},...,\tau_{n_{s-1}+2}$, we get that 
\begin{align*}
&\mathcal{K}^{n_{s-1},n_s}(\underline{\alpha}_d)\\
&=\frac{\Gamma^{n+1-n_s}((d+3)/2)\Gamma^{n_s-n_{s-1}-1}((d+1)/2)}{(c||\underline{\alpha}_d||)^{(d+1)(n-n_{s-1})-1+d}(\sqrt{2\pi})^{n-n_{s-1}-1}\Gamma(\frac{(n-n_{s-1})}{2}(d+3)-(n_s-n_{s-1}-1))}\\
&\int_0^{t-\sum_{k=1}^{n_{s-1}}\tau_k}d\tau_{n_{s-1}+1}(c\tau_{n_{s-1}+1}||\underline{\alpha}_d||)^{d/2}(c(t-\sum_{k=1}^{n_{s-1}+1}\tau_k)||\underline{\alpha}_d||)^{\frac{(n-n_{s-1})}{2}(d+3)-(n_s-n_{s-1}-\frac12)}\\
&\quad J_{d/2}(c\tau_{n_{s-1}+1}||\underline{\alpha}_d||)J_{\frac{(n-n_{s-1})}{2}(d+3)-(n_s-n_{s-1}-\frac12)}(c(t-\sum_{k=1}^{n_{s-1}+1}\tau_k)||\underline{\alpha}_d||)\\
&=(c\tau_{n_{s-1}+1}||\underline{\alpha}_d||=w)\\
&=\frac{\Gamma^{n+1-n_s}((d+3)/2)\Gamma^{n_s-n_{s-1}-1}((d+1)/2)}{(c||\underline{\alpha}_d||)^{(d+1)(n-n_{s-1}+1)-1}(\sqrt{2\pi})^{n-n_{s-1}-1}\Gamma(\frac{(n-n_{s-1})}{2}(d+3)-(n_s-n_{s-1}-1))}\\
&\quad\int_0^{c(t-\sum_{k=1}^{n_{s-1}}\tau_k)||\underline{\alpha}_d||}dww^{d/2}(c(t-\sum_{k=1}^{n_{s-1}}\tau_k)||\underline{\alpha}_d||-w)^{\frac{(n-n_{s-1})}{2}(d+3)-(n_s-n_{s-1}-\frac12)}\\
&\qquad J_{d/2}(w)J_{\frac{(n-n_{s-1})}{2}(d+3)-(n_s-n_{s-1}-\frac12)}(c(t-\sum_{k=1}^{n_{s-1}}\tau_k)||\underline{\alpha}_d||-w)\\
&=\frac{((d+1)/2)^{n+1-n_{s}}\Gamma^{n+1-n_{s-1}}((d+1)/2)}{(c||\underline{\alpha}_d||)^{(n+1-n_{s-1})(d+1)-1}(\sqrt{2\pi})^{n-n_{s-1}}\Gamma(\frac{(n+1-n_{s-1})(d+3)}{2}-(n_s-n_{s-1}))}\\
&\quad(c(t-\sum_{k=1}^{n_{s-1}}\tau_k)||\underline{\alpha}_d||)^{\frac{(n+1-n_{s-1})(d+3)}{2}-(n_s-n_{s-1}+1/2)}J_{\frac{(n+1-n_{s-1})(d+3)}{2}-(n_s-n_{s-1}+1/2)}(c(t-\sum_{k=1}^{n_{s-1}}\tau_k)||\underline{\alpha}_d||)
\end{align*} 
where in the last step we have also exploited the formula $\Gamma(1+z)=z\Gamma(z)$.

Now, bearing in mind the same steps used to work out $\mathcal{K}^{n_s}(\underline{\alpha}_d)$, we are able to explicit, omitting the calculations, the following integral 

\begin{align*}
&\mathcal{K}^{n_{s-2},n_{s-1},n_s}(\underline{\alpha}_d)\\
 &=\int_0^{t-\sum_{k=1}^{n_{s-2}}\tau_k}\tau_{n_{s-2}+1}^{d}d\tau_{n_{s-2}+1}\int_0^{t-\sum_{k=1}^{n_{s-2}+1}\tau_k}\tau_{{n_{s-2}+2}}^dd\tau_{{n_{s-2}+2}}\cdots\\
 &\quad\int_0^{t-\sum_{k=1}^{n_{s-1}-2}\tau_k}\tau_{n_{s-1}-1}^dd\tau_{n_{s-1}-1}\int_0^{t-\sum_{k=1}^{n_{s-1}-1}\tau_k}\tau_{n_{s-1}}^dd\tau_{n_{s-1}}\mathcal{K}^{n_{s-1},n_s}(\underline{\alpha}_d)\prod_{k=n_{s-2}+1}^{n_{s-1}}\frac{J_{d/2+1}(c\tau_k||\underline{\alpha}_d||)}{(c\tau_k||\underline{\alpha}_d||)^{d/2-1}}\\
 &=\frac{((d+1)/2)^{n+1-n_{s}+n_{s-1}-n_{s-2}}\Gamma^{n+1-n_{s-2}}((d+1)/2)}{(c||\underline{\alpha}_d||)^{(n+1-n_{s-2})(d+1)-1}(\sqrt{2\pi})^{n-n_{s-2}}\Gamma(\frac{(n+1-n_{s-2})(d+3)}{2}-(n_s-n_{s-1}))}\\
&\quad(c(t-\sum_{k=1}^{n_{s-2}}\tau_k)||\underline{\alpha}_d||)^{\frac{(n+1-n_{s-2})(d+3)}{2}-(n_s-n_{s-1}+1/2)}J_{\frac{(n+1-n_{s-2})(d+3)}{2}-(n_s-n_{s-1}+1/2)}(c(t-\sum_{k=1}^{n_{s-2}}\tau_k)||\underline{\alpha}_d||).
\end{align*}

Analogously to $\mathcal{K}^{n_{s-1},n_s}(\underline{\alpha}_d)$, we get that
\begin{align*}
&\mathcal{K}^{n_{s-3},n_{s-2},n_{s-1},n_s}(\underline{\alpha}_d)\\
 &=\int_0^{t-\sum_{k=1}^{n_{s-3}}\tau_k}\tau_{n_{s-3}+1}^{d}d\tau_{n_{s-3}+1}\int_0^{t-\sum_{k=1}^{n_{s-3}+1}\tau_k}\tau_{{n_{s-3}+2}}^dd\tau_{{n_{s-3}+2}}\cdots\\
 &\quad\int_0^{t-\sum_{k=1}^{n_{s-2}-2}\tau_k}\tau_{n_{s-2}-1}^dd\tau_{n_{s-2}-1}\int_0^{t-\sum_{k=1}^{n_{s-2}-1}\tau_k}\tau_{n_{s-2}}^dd\tau_{n_{s-2}}\mathcal{K}^{n_{s-2},n_{s-1},n_s}(\underline{\alpha}_d)\prod_{k=n_{s-3}+1}^{n_{s-2}}\frac{J_{d/2}(c\tau_k||\underline{\alpha}_d||)}{(c\tau_k||\underline{\alpha}_d||)^{d/2}}\\
 &=\frac{((d+1)/2)^{n+1-n_{s}+n_{s-1}-n_{s-2}}\Gamma^{n+1-n_{s-3}}((d+1)/2)}{(c||\underline{\alpha}_d||)^{(n+1-n_{s-3})(d+1)-1}(\sqrt{2\pi})^{n-n_{s-3}}\Gamma(\frac{(n+1-n_{s-3})(d+3)}{2}-[(n_s-n_{s-1})+(n_{s-2}-n_{s-3})])}\\
&\quad(c(t-\sum_{k=1}^{n_{s-3}}\tau_k)||\underline{\alpha}_d||)^{\frac{(n+1-n_{s-3})(d+3)}{2}-[(n_s-n_{s-1})+(n_{s-2}-n_{s-3})+1/2]}\\
&\quad\times J_{\frac{(n+1-n_{s-3})(d+3)}{2}-[(n_s-n_{s-1})+(n_{s-2}-n_{s-3})+1/2]}(c(t-\sum_{k=1}^{n_{s-3}}\tau_k)||\underline{\alpha}_d||).
\end{align*}

Finally, by applying recursively the same arguments exploited so far, we get that

\begin{align*}
  &\mathcal{J}_{n,j}^{n_1,n_2,...,n_s}(\underline{\alpha}_d)\\
 &=\left\{\frac{2^{d/2}\Gamma(1+d/2)}{\Gamma(d+1)}\right\}^{n+1}\frac{\Gamma((n+1)(d+1))}{t^{(n+1)(d+1)-1}}\left(-\frac{\alpha_{d}^2}{||\underline{\alpha}_d||^2}\right)^{n+1-j}\notag\\
 &\quad\int_0^{t}\tau_1^{d}d\tau_1\int_0^{t-\tau_1}\tau_2^dd\tau_2\cdots\int_0^{t-\sum_{k=1}^{n_1-2}}\tau_{n_1-1}^dd\tau_{n_1-1}\int_0^{t-\sum_{k=1}^{n_1-1}\tau_k}\tau_{n_1}^dd\tau_{n_1}\prod_{k=1}^{n_1}\frac{J_{d/2}(c\tau_k||\underline{\alpha}_d||)}{(c\tau_k||\underline{\alpha}_d||)^{d/2}}\notag\\
 &\quad\int_0^{t-\sum_{k=1}^{n_1}\tau_{n_1+1}}\tau_{n_1+1}^{d}d\tau_{n_1+1}\int_0^{t-\sum_{k=1}^{n_1+1}\tau_{n}}\tau_{n_1+2}^dd\tau_{n_1+2}\cdots\notag\\
 &\quad\int_0^{t-\sum_{k=1}^{n_2-2}\tau_k}\tau_{n_2-1}^dd\tau_{n_2-1}\int_0^{t-\sum_{k=1}^{n_2-1}\tau_k}\tau_{n_2}^dd\tau_{n_2}\prod_{k=n_1+1}^{n_2}\frac{J_{d/2+1}(c\tau_k||\underline{\alpha}_d||)}{(c\tau_k||\underline{\alpha}_d||)^{d/2-1}}\\
 &\quad\frac{((d+1)/2)^{n+1-[(n_s-n_{s-1})+(n_{s-2}-n_{s-3})+...+(n_{3}-n_{2})]}\Gamma^{n+1-n_{2}}((d+1)/2)}{(c||\underline{\alpha}_d||)^{(n+1-n_{2})(d+1)-1}(\sqrt{2\pi})^{n-n_{2}}\Gamma(\frac{(n+1-n_{2})(d+3)}{2}-[(n_s-n_{s-1})+(n_{s-2}-n_{s-3})+...+(n_{3}-n_{2})])}\\
&\quad(c(t-\sum_{k=1}^{n_{2}}\tau_k)||\underline{\alpha}_d||)^{\frac{(n+1-n_{2})(d+3)}{2}-[(n_s-n_{s-1})+(n_{s-2}-n_{s-3})+...+(n_{3}-n_{2})+1/2]}\\
&\quad J_{\frac{(n+1-n_{2})(d+3)}{2}-[(n_s-n_{s-1})+(n_{s-2}-n_{s-3})+...+(n_{3}-n_{2})+1/2]}(c(t-\sum_{k=1}^{n_{2}}\tau_k)||\underline{\alpha}_d||)\\
&=\left\{\frac{2^{d/2}\Gamma(1+d/2)}{\Gamma(d+1)}\right\}^{n+1}\frac{\Gamma((n+1)(d+1))}{t^{(n+1)(d+1)-1}}\left(-\frac{\alpha_{d}^2}{||\underline{\alpha}_d||^2}\right)^{n+1-j}\\
&\quad\frac{((d+1)/2)^{n+1-j}\Gamma^{n+1}((d+1)/2)}{(c||\underline{\alpha}_d||)^{(n+1)(d+1)-1}(\sqrt{2\pi})^n\Gamma(\frac{(n+1)(d+3)}{2}-j)}(ct||\underline{\alpha}_d||)^{\frac{(n+1)(d+3)}{2}-j-\frac12}J_{\frac{(n+1)(d+3)}{2}-(j+\frac12)}(ct||\underline{\alpha}_d||)\\
&= \frac{\sqrt{\pi}\Gamma((n+1)(d+1))}{2^{\frac{(d+1)(n+1)-1}{2}}\Gamma(\frac{(n+1)(d+3)}{2}-j)}\frac{\left(-\frac{\alpha_{d}^2}{||\underline{\alpha}_d||^2}((d+1)/2)\right)^{n+1-j}}{(ct||\underline{\alpha}_d||)^{\frac{(n+1)(d-1)+2j-1}{2}}}J_{\frac{(n+1)(d+3)-(2j+1)}{2}}(ct||\underline{\alpha}_d||).
   \end{align*}
 
The integrals $\mathcal{J}_{n,j}^{n_1,n_2,...,n_s}(\underline{\alpha}_d)$ have different configurations involving the Bessel functions $J_{\frac{d}{2}}(c\tau_k||\underline{\alpha}_d||)$ and  $J_{\frac{d}{2}-1}(c\tau_k||\underline{\alpha}_d||)$, according to the values $n_1,n_2,...,n_s$. Nevertheless, from the previous calculations, we observed that the value of $\mathcal{J}_{n,j}^{n_1,n_2,...,n_s}(\underline{\alpha}_d)$ is the same if the $j$ does not change for different combinations of  $n_1,n_2,...,n_s$. In other words, in order to calculate the values of the different terms of the sum appearing in \eqref{eq:cfintegralsum}, we have to consider the number of Bessel functions with index $\frac d2$ (or $\frac d2-1$) appearing in \eqref{eq:mixintgen}. Then, we can conclude that
   \begin{equation}
   \mathcal{F}_n^1(\underline{\alpha}_d)=\sum_{j=0}^{n+1}\binom{n+1}{j}\mathcal{J}_{n,j}^{n_1,n_2,...,n_s}(\underline{\alpha}_d)
   \end{equation}
and the proof is completed.
\end{proof}

\begin{remark}
The analytic form of the characteristic function $\mathcal{F}_n^1(\underline{\alpha}_d)$ is more complicated than its counterpart in the uniform case ($\nu=0$) (see formula (2.1) in De Gregorio and Orsingher, 2011) which is only given by a Bessel function (with a suitable constant). This is because  the random motions considered in this paper have drift. Indeed, the choice of directions non-uniformly distributed, for each step performed by the motion, implies the loss of spherical symmetry. 
\end{remark}
\begin{remark}
By means of \eqref{eq:int2} and \eqref{eq:int3}, the characteristic function of $\underline{\bf X}_d(t),t>0,$ could be calculated for $\nu=2,3$. Clearly, in these cases the necessary calculations for completing the proof  follow the same steps of the proof of Theorem \ref{teo:cfnu1} are very long and cumbersome. Indeed, the expression \eqref{eq:cfintegralsum} involves the product of linear combinations of Bessel functions.  
\end{remark}

Theorem \ref{teo:cfnu1} allows to show our next result. We are able to provide the density law of $\underline{\bf X}_d(t),t>0,$ (up to some positive constants).  
\begin{theorem}
The density function of $\underline{\bf X}_d(t),t>0,$ for $\nu=1$ assumes the following form
\begin{align}\label{eq:densnu1}
p_n^1(\underline{\bf x}_d,t)&=\frac{\Gamma((n+1)(d+1))}{\pi^{\frac{d-1}{2}}(2ct)^{(n+1)(d+1)-1}}\sum_{j=0}^{n+1}(-1)^{n+1-j}\binom{n+1}{j}\frac{\left(\frac{d+1}{2}\right)^{n+1-j}}{\Gamma(\frac{(n+1)(d+3)}{2}-j)}\notag\\
&\quad\sum_{k=0}^{n+1-j}\frac{(-1)^ka_{k,n+1-j}x_d^{2k}}{\Gamma\left(\frac{n(d+1)}{2}-k+1\right)}(c^2t^2-||\underline{\bf x}_d||^2)^{\frac{n(d+1)}{2}-k}
\end{align}
where $||\underline{\bf x}_d||<ct$ and $a_{k,n+1-j}$s are positive constants defined as in \eqref{eq:intgeneral}. 
\end{theorem}
\begin{proof}
We invert the characteristic function \eqref{eq:cfnu1} as follows
\begin{align*}
&p_n^1(\underline{\bf x}_d,t)\\
&=\frac{1}{(2\pi)^{d}}\int_{\mathbb{R}^d}e^{-i<\underline{\alpha}_d,\underline{\bf x}_d>}\mathcal{F}_n^1(\underline{\alpha}_d)\prod_{i=1}^dd\alpha_i\\
&=\frac{\sqrt{\pi}\Gamma((n+1)(d+1))}{(2\pi)^{d}2^{\frac{(d+1)(n+1)-1}{2}}}\sum_{j=0}^{n+1}(-1)^{n+1-j}\binom{n+1}{j}\frac{\left(\frac{d+1}{2}\right)^{n+1-j}}{(ct)^{\frac{(n+1)(d-1)+2j-1}{2}}\Gamma(\frac{(n+1)(d+3)}{2}-j)}\\
&\quad\int_{\mathbb{R}^d}e^{-i<\underline{\alpha}_d,\underline{\bf x}_d>}\left(\frac{\alpha_{d}^2}{||\underline{\alpha}_d||^2}\right)^{n+1-j}\frac{J_{\frac{(n+1)(d+3)-(2j+1)}{2}}(ct||\underline{\alpha}_d||)}{||\underline{\alpha}_d||^{\frac{(n+1)(d-1)+2j-1}{2}}}\prod_{i=1}^dd\alpha_i\\
&=\frac{\sqrt{\pi}\Gamma((n+1)(d+1))}{(2\pi)^{d}2^{\frac{(d+1)(n+1)-1}{2}}}\sum_{j=0}^{n+1}(-1)^{n+1-j}\binom{n+1}{j}\frac{\left(\frac{d+1}{2}\right)^{n+1-j}}{(ct)^{\frac{(n+1)(d-1)+2j-1}{2}}\Gamma(\frac{(n+1)(d+3)}{2}-j)}\\
&\quad\int_0^\infty \rho^{d-1}d\rho\int_0^\pi d\theta_1\cdots\int_0^{\pi}d\theta_{d-2}\int_0^{2\pi}d\phi e^{-i\rho(x_d\sin\theta_1\cdots\sin\phi+x_{d-1}\sin\theta_1\cdots\cos\phi+...+x_2\sin\theta_1\cos\theta_2+x_1\cos\theta_1)}\\
&\quad\left(\sin\theta_1\cdots\sin\theta_{d-2}\sin\phi\right)^{2(n+1-j)}\frac{J_{\frac{(n+1)(d+3)-(2j+1)}{2}}(ct\rho)}{\rho^{\frac{(n+1)(d-1)+2j-1}{2}}}\sin\theta_1^{d-2}\cdots\sin\theta_{d-2}.
\end{align*}

By taking into account the formula \eqref{eq:intgeneral}, the integral with respect to $\phi$ becomes
\begin{align*}
&\int_0^{2\pi} e^{-i\rho(x_d\sin\theta_1\cdots\sin\theta_{d-2}\sin\phi+x_{d-1}\sin\theta_1\cdots\sin\theta_{d-2}\cos\phi)}\sin\phi^{2(n+1-j)}d\phi\\
&=2\pi\sum_{k=0}^{n+1-j}\frac{(-1)^ka_{k,n+1-j}x_d^{2k}J_{n+1-j+k}(\rho\sin\theta_1\cdots\sin\theta_{d-2}\sqrt{x_d^2+x_{d
-1}^2})
}{2^k(\rho\sin\theta_1\cdots\sin\theta_{d-2})^{n+1-j-k}(\sqrt{x_d^2+x_{d
-1}^2})^{n+1-j+k} }.
\end{align*}
Then
\begin{align*}
&p_n^1(\underline{\bf x}_d,t)\\
&=\frac{\sqrt{\pi}\Gamma((n+1)(d+1))}{(2\pi)^{d-1}2^{\frac{(d+1)(n+1)-1}{2}}}\sum_{j=0}^{n+1}(-1)^{n+1-j}\binom{n+1}{j}\frac{\left(\frac{d+1}{2}\right)^{n+1-j}}{(ct)^{\frac{(n+1)(d-1)+2j-1}{2}}\Gamma(\frac{(n+1)(d+3)}{2}-j)}\\
&\quad\sum_{k=0}^{n+1-j}\frac{(-1)^ka_{k,n+1-j}x_d^{2k}}{2^k}\int_0^\infty \rho^{d-1-(n+1-j-k)}\frac{J_{\frac{(n+1)(d+3)-(2j+1)}{2}}(ct\rho)}{\rho^{\frac{(n+1)(d-1)+2j-1}{2}}}d\rho\\
&\quad\int_0^\pi d\theta_1\cdots\int_0^{\pi}d\theta_{d-2}e^{-i\rho(x_{d-2}\sin\theta_1\cdots\sin\theta_{d-3}\cos\theta_{d-2}+...+x_1\cos\theta_1)}\sin\theta_1^{d-2}\cdots\sin\theta_{d-2}\\
&\quad \frac{\left(\sin\theta_1\cdots\sin\theta_{d-2}\right)^{n+1-j+k}}{(\sqrt{x_d^2+x_{d
-1}^2})^{n+1-j+k}}J_{n+1-j+k}(\rho\sin\theta_1\cdots\sin\theta_{d-2}\sqrt{x_d^2+x_{d
-1}^2})
\end{align*}

By means of the same approach used to get \eqref{intangle}, the $(d-2)$-fold integral with respect to the angles becomes
\begin{align*}
&\int_0^\pi d\theta_1\cdots\int_0^{\pi}d\theta_{d-2}e^{-i\rho(x_{d-2}\sin\theta_1\cdots\sin\theta_{d-3}\cos\theta_{d-2}+...+x_1\cos\theta_1)}\sin\theta_1^{d-2}\cdots\sin\theta_{d-2}\\
&\quad \frac{\left(\sin\theta_1\cdots\sin\theta_{d-2}\right)^{n+1-j+k}}{(\sqrt{x_d^2+x_{d
-1}^2})^{n+1-j+k}}J_{n+1-j+k}(\rho\sin\theta_1\cdots\sin\theta_{d-2}\sqrt{x_d^2+x_{d
-1}^2})\\
&=\frac{(2\pi)^{\frac d2-1}}{\rho^{\frac d2-1}||\underline{\bf x}_d||^{n+1-j+k+\frac d2-1}}J_{n+1-j+k+\frac d2-1}(\rho ||\underline{\bf x}_d||).
\end{align*}

Therefore, the formula \eqref{formula3} implies that
\begin{align*}
&p_n^1(\underline{\bf x}_d,t)\\
&=\frac{\sqrt{\pi}\Gamma((n+1)(d+1))}{(2\pi)^{\frac d2}2^{\frac{(d+1)(n+1)-1}{2}}}\sum_{j=0}^{n+1}(-1)^{n+1-j}\binom{n+1}{j}\frac{\left(\frac{d+1}{2}\right)^{n+1-j}}{(ct)^{\frac{(n+1)(d-1)+2j-1}{2}}\Gamma(\frac{(n+1)(d+3)}{2}-j)}\\
&\quad\sum_{k=0}^{n+1-j}\frac{(-1)^ka_{k,n+1-j}}{2^k}\frac{x_d^{2k}}{||\underline{\bf x}_d||^{n+1-j+k+\frac d2-1}}\int_0^\infty \frac{J_{\frac{(n+1)(d+3)-(2j+1)}{2}}(ct\rho)J_{n+1-j+k+\frac d2-1}(\rho ||\underline{\bf x}_d||)}{\rho^{n\frac{(d+1)}{2}-k}}d\rho\\
&=(\text{by}\, \eqref{formula3})\\
&=\frac{\sqrt{\pi}\Gamma((n+1)(d+1))}{(2\pi)^{\frac d2}2^{\frac{(d+1)(n+1)-1}{2}}}\sum_{j=0}^{n+1}(-1)^{n+1-j}\binom{n+1}{j}\frac{\left(\frac{d+1}{2}\right)^{n+1-j}}{(ct)^{\frac{(n+1)(d-1)+2j-1}{2}}\Gamma(\frac{(n+1)(d+3)}{2}-j)}\\
&\quad\sum_{k=0}^{n+1-j}\frac{(-1)^ka_{k,n+1-j}}{2^{\frac{n(d+1)}{2}}}\frac{x_d^{2k}}{\Gamma\left(\frac{n(d+1)}{2}-k+1\right)(ct)^{\frac{(n+1)(d+3)-(2j+1)}{2}}}(c^2t^2-||\underline{\bf x}_d||^2)^{\frac{n(d+1)}{2}-k}\\
&=\frac{\Gamma((n+1)(d+1))}{\pi^{\frac{d-1}{2}}(2ct)^{(n+1)(d+1)-1}}\sum_{j=0}^{n+1}(-1)^{n+1-j}\binom{n+1}{j}\frac{\left(\frac{d+1}{2}\right)^{n+1-j}}{\Gamma(\frac{(n+1)(d+3)}{2}-j)}\\
&\quad\sum_{k=0}^{n+1-j}\frac{(-1)^ka_{k,n+1-j}x_d^{2k}}{\Gamma\left(\frac{n(d+1)}{2}-k+1\right)}(c^2t^2-||\underline{\bf x}_d||^2)^{\frac{n(d+1)}{2}-k}.
\end{align*}
\end{proof}

It is worth to point out again that the result stated in the above Theorem does not provide the whole analytic form of the density law of the random flight with $\nu=1$. Indeed the constants appearing in the expression \eqref{eq:densnu1} are not determined for an arbitrary value of $n$, but only for some values of $n$ and $n+1-j$. For instance $a_{0,n+1-j}=(2(n+1-j)-1)(2(n+1-j)-3)\cdots 3\cdot 1$ and $a_{n+1-j,n+1-j}=2^{n+1-j}$  (see formula \eqref{eq:intgeneral} in Appendix). Furthermore, in a few cases we can explicit $p_n^1(\underline{\bf x}_d,t)$. By taking into account \eqref{const}, for $n=1$ and $n=2$, after some calculations, we obtain that
\begin{align}\label{eq:rannonun1}
p_1^1(\underline{\bf x}_d,t)&=\frac{\Gamma(2(d+1))}{\pi^{\frac{d-1}{2}}(2ct)^{2d+1}(d+2)\Gamma(d+1)\Gamma(\frac{d-1}{2})}\notag\\
&\quad\left\{\frac{3}{d-1}(c^2t^2-||\underline{\bf x}_d||^2)^{\frac{d+1}{2}}-2x_d^2(c^2t^2-||\underline{\bf x}_d||^2)^{\frac{d-1}{2}}+(d+1)x_d^4(c^2t^2-||\underline{\bf x}_d||^2)^{\frac{d-3}{2}}\right\}
\end{align}
and
\begin{align}\label{eq:rannonun2}
p_2^1(\underline{\bf x}_d,t)&=\frac{\Gamma(3d+3)(d+1)}{\pi^{\frac{d-1}{2}}(2ct)^{3d+2}\Gamma(d-1)\Gamma(\frac{3(d+3)}{2}-3)(3d+7)(3d+5)}\notag\\
&\quad\Bigg\{\frac{4(d+4)}{(d+1)d(d-1)}(c^2t^2-||\underline{\bf x}_d||^2)^{d+1}+\frac{2x_d^2(6d^2+6d+8)}{(d+1)d(d-1)}(c^2t^2-||\underline{\bf x}_d||^2)^d\notag\\
&\quad-8 x_d^4(c^2t^2-||\underline{\bf x}_d||^2)^{d-1}+\frac{8}{3}(d+1)x_d^6(c^2t^2-||\underline{\bf x}_d||^2)^{d-2}\Bigg\}
\end{align}

Let $f_{n,d}^\nu(r,t)$ be the density function of $R_d(t)=||\underline{\bf X}_d(t)||,t>0,$ with $0<r<ct$. As done for obtaining \eqref{densradial}, from \eqref{eq:rannonun1} and \eqref{eq:rannonun2} we get that
\begin{align}\label{eq:rannonunrad1}
f_{1,d}^1(r,t)&=\frac{\Gamma(2(d+1))2\sqrt{\pi}}{(2ct)^{2d+1}(d+2)\Gamma(d+1)\Gamma(\frac{d-1}{2})\Gamma(\frac d2)}\notag\\
&\quad\left\{\frac{3}{d-1}r^{d-1}(c^2t^2-r^2)^{\frac{d+1}{2}}-\frac 2dr^{d+1}(c^2t^2-r^2)^{\frac{d-1}{2}}+\frac{3(d+1)}{d(d+2)}r^{d+3}(c^2t^2-r^2)^{\frac{d-3}{2}}\right\}
\end{align}
and
\begin{align}\label{eq:rannonunrad2}
f_{2,d}^1(r,t)&=\frac{\Gamma(3d+3)(d+1)2\sqrt{\pi}}{(2ct)^{3d+2}\Gamma(d-1)\Gamma(\frac{3(d+3)}{2}-3)\Gamma(\frac{d}{2})(3d+7)(3d+5)}\notag\\
&\quad\Bigg\{\frac{4(d+4)}{(d+1)d(d-1)}r^{d-1}(c^2t^2-r^2)^{d+1}+\frac{2(6d^2+6d+8)}{(d+1)d^2(d-1)}r^{d+1}(c^2t^2-r^2)^d\notag\\
&\quad-\frac{24}{d(d+2)} r^{d+3}(c^2t^2-r^2)^{d-1}+\frac{40(d+1)}{d(d+2)(d+4)}r^{d+5}(c^2t^2-r^2)^{d-2}\Bigg\}
\end{align}

\begin{remark}
The expressions \eqref{eq:rannonun1} and \eqref{eq:rannonun2} represent proper density functions. Indeed, it is possible to verify after some calculations that $\int_{\mathcal{H}_{ct}^d} p_n^1(\underline{\bf x}_d,t)d\underline{\bf x}_d=1,n=1,2$. Furthermore, $p_1^1(\underline{\bf x}_d,t)$ and $p_2^1(\underline{\bf x}_d,t)$, or equivalently $f_{1,d}^1(r,t)$ and $f_{2,d}^1(r,t)$, are non negative functions (see Figure \ref{radial}).

 \begin{figure}[t]
 \begin{center}
\includegraphics[angle=0,width=1\textwidth]{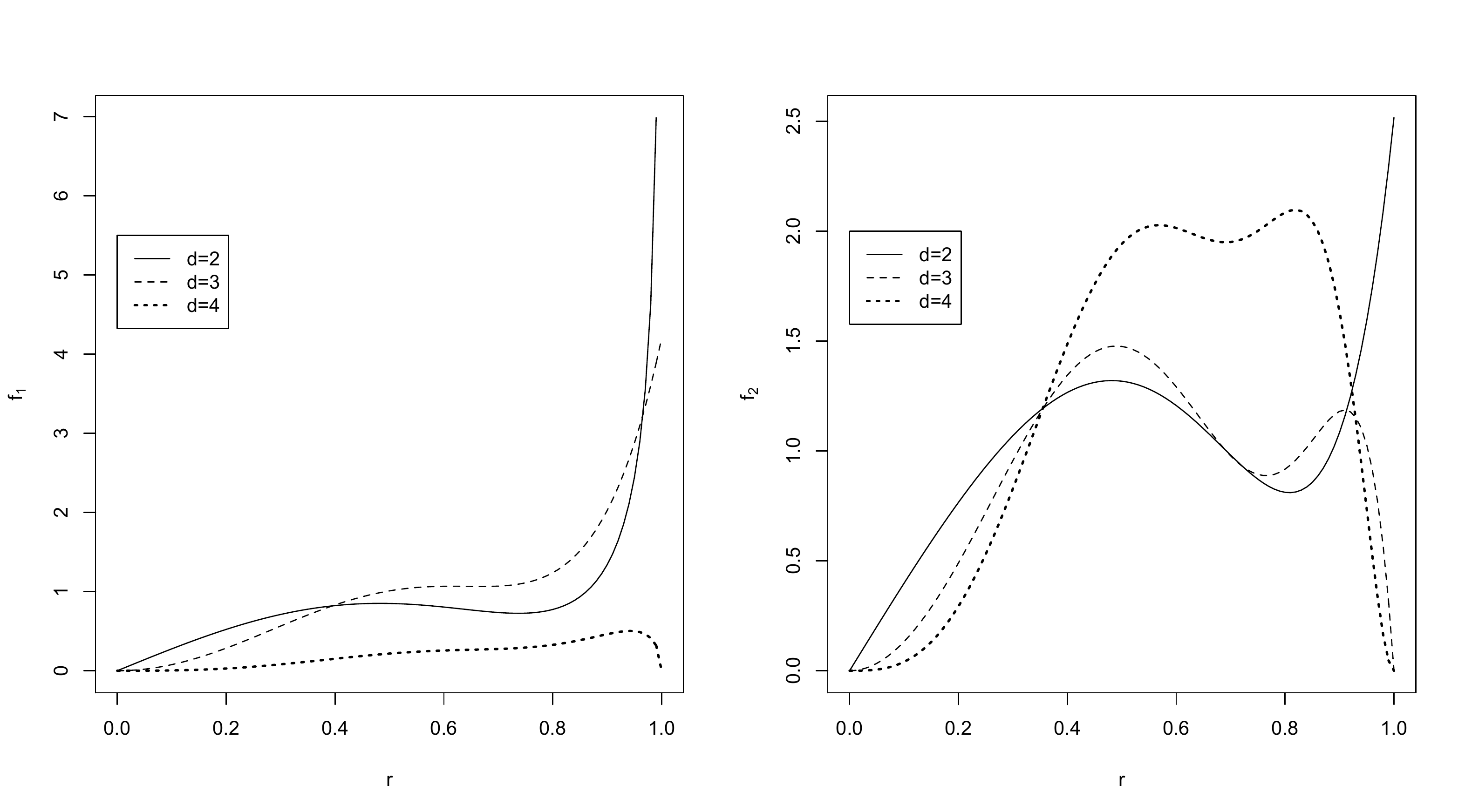}
\caption{The behavior of the density functions $f_1=f_{1,d}^1(r,1)$ (on the left) and $f_2=f_{2,d}^1(r,1)$ (on the right) for $d=2,3,4$ and $c=1$.}\label{radial}
\end{center}
\end{figure}

\end{remark}

\begin{remark}
It is useful to compare the density functions $p_1^1(\underline{\bf x}_d,t)$ and $p_2^1(\underline{\bf x}_d,t)$ with the probability distributions obtained in the uniform case ($\nu=0$) by Orsingher and De Gregorio (2007). As emerges from Table \ref{tab:first}-\ref{tab:second}, the uniform random flight and the random motion with drift  lead to quite different distributions.

\begin{table}[h] 
\small{
\begin{tabular}{c|l|l }
&$\nu=0$&$\nu=1$\\
\hline
    $d=2$     &$\frac{1}{2\pi ct\sqrt{c^2t^2-||\underline{\bf x}_2||^2}}$& $\frac{15}{2^4\pi(ct)^5}\left\{\frac32(c^2t^2-||\underline{\bf x}_2||^2)^{\frac32}
-x_2^2\sqrt{c^2t^2-||\underline{\bf x}_2||^2}+\frac{3x_2^4}{2\sqrt{c^2t^2-||\underline{\bf x}_2||^2}}\right\}$\\&&\\
       $d=3$     &$\frac{\log(\frac{c^2t^2+||\underline{\bf x}_3||}{c^2t^2-||\underline{\bf x}_3||})}{\pi (2ct)^2||\underline{\bf x}_3||}$    &$\frac{21}{2^4\pi(ct)^7}\left\{\frac32(c^2t^2-||\underline{\bf x}_3||^2)^{2}
-2x_3^2(c^2t^2-||\underline{\bf x}_3||^2)+4x_3^4\right\}$ \\&&\\
      $d=4$    &$\frac{2}{\pi^2(ct)^4}$&$\frac{3^2\cdot 7\cdot 5}{2^5\pi^2(ct)^9}\left\{(c^2t^2-||\underline{\bf x}_4||^2)^{\frac52}
-2x_4^4(c^2t^2-||\underline{\bf x}_4||^2)^{\frac32}+5x_4^4\sqrt{c^2t^2-||\underline{\bf x}_4||^2}\right\}$    
      \end{tabular}
   \caption{Density function $p_1^\nu(\underline{\bf x}_d,t)$ for $\nu=0,1,$ and $d=2,3,4$.}
   \label{tab:first}}
\end{table}
\begin{table}[h] 
\small{
\begin{tabular}{c|l|l }
&$\nu=0$&$\nu=1$\\
\hline
    $d=2$     &$\frac{1}{\pi (ct)^2}$& $\frac{2^5\cdot 3^2}{\pi(ct)^813\cdot 11}\Bigg\{(c^2t^2-||\underline{\bf x}_2||^2)^{3}
+\frac{11}{3}x_2^2(c^2t^2-||\underline{\bf x}_2||^2)^2$\\
&&$-2x_2^4(c^2t^2-||\underline{\bf x}_2||^2)+2x_2^6\Bigg\}$\\&&\\
      $d=4$    &$\frac{6}{\pi^2(ct)^6}(c^2t^2-||\underline{\bf x}_4||^2)$&  $\frac{2^6\cdot 3^2\cdot 7\cdot 5}{\pi^2(ct)^{14}19\cdot 17}\Bigg\{\frac{1}{15}(c^2t^2-||\underline{\bf x}_4||^2)^{5}
+\frac{8}{15}x_4^2(c^2t^2-||\underline{\bf x}_4||^2)^4$\\
&&$-x_4^4(c^2t^2-||\underline{\bf x}_4||^2)^3+\frac53x_4^6(c^2t^2-||\underline{\bf x}_4||^2)^2\Bigg\}$   
      \end{tabular}
   \caption{Density function $p_2^\nu(\underline{\bf x}_d,t)$ for $\nu=0,1,$ and $d=2,4$.}
   \label{tab:second}}
\end{table}

\end{remark}

\begin{remark}
From the results \eqref{eq:densnu1}, \eqref{eq:rannonun1} and \eqref{eq:rannonun2} emerge that the probability distributions of the random flight with $\nu=1$ are a linear combination of $n+2$ functions 
$$x_d^{2k}(c^2t^2-||\underline{\bf x}_d||^2)^{\frac{n(d+1)}{2}-k}$$
with $k=0,1,...,n+1$. Furthermore, \eqref{eq:densnu1} permits to claim that, as expected, the random flight with drift is not isotropic. It is interesting to point out that the observer in $\mathbb{R}^m,m<d$, perceives the random motion as an isotropic walk (see Remark \ref{rem}), while the original motion is non-isotropic. 
\end{remark}

\section{Appendix}
In the random flight problem the integrals involving Bessel functions have a crucial role. For this reason in this Section we summarize some important results which are often used in the manuscript.

In particular we focus our attention on the following integral 
\begin{equation}\label{eq:intgen}
\int_0^{2\pi}e^{iz(\alpha\cos\theta+\beta\sin\theta)}\sin^{2\nu}\theta d\theta
\end{equation}
with $\alpha,\beta\in\mathbb{R}$ and $\nu\geq 0$. We are able to calculate \eqref{eq:intgen} in some particular cases. For example, for $\nu=0$ it is well-known that

\begin{equation}\label{eq:int0}
\int_0^{2\pi}e^{iz(\alpha\cos\theta+\beta\sin\theta)}d\theta=2\pi J_0(z\sqrt{\alpha^2+\beta^2}).
\end{equation}

If $\nu=1$, we get that
\begin{equation}\label{eq:int}
\int_0^{2\pi}e^{iz(\alpha\cos\theta+\beta\sin\theta)}\sin^{2}\theta d\theta=2\pi\left\{\frac{J_1(z\sqrt{\alpha^2+\beta^2})}{z\sqrt{\alpha^2+\beta^2}}-\frac{\beta^2}{\alpha^2+\beta^2}J_2(z\sqrt{\alpha^2+\beta^2})\right\}.
\end{equation}
as proved in
 De Gregorio (2010). The previous result has been obtained by expanding the exponential function inside the integral. Hence, one observes that splitting the integral into two parts, i.e. $\int_0^{2\pi}=\int_0^{\pi}+\int_\pi^{2\pi}$, after a change of variable, the sum is not equal to zero if the index values are even.
 
Following the same approach developed in De Gregorio (2010) we can write that
\begin{align*}
&\int_0^{2\pi}\exp\{iz(\alpha\cos\theta+\beta\sin\theta)\}\sin^{4}\theta d\theta\notag\\
&=4\sum_{k=0}^{\infty} \frac{(-1)^k}{(2k!)}z^{2k}\sum_{r=0}^k\binom{2k}{2r}\alpha^{2r}\beta^{2(k-r)}\int_0^{\pi/2}\cos^{2r}\theta\sin^{2(k-r+2)}\theta d\theta\notag\\
&=2\sum_{k=0}^{\infty} \frac{(-1)^k}{(2k!)}z^{2k}\sum_{r=0}^k\binom{2k}{2r}\alpha^{2r}\beta^{2(k-r)}\frac{\Gamma(r+1/2)\Gamma(k-r+2+1/2)}{\Gamma(k+3)}\notag\\
&=2\pi\sum_{k=0}^{\infty} \frac{(-1)^k}{(2k!)}\frac{z^{2k}}{\Gamma(k+3)}\sum_{r=0}^k\frac{(2k)!}{(2r)!(2(k-r))!}\alpha^{2r}\beta^{2(k-r)}\frac{\Gamma(2r)\Gamma(2(k-r+2))2^{2-2(k+2)}}{\Gamma(r)\Gamma(k-r+2)}\notag\\
&=\frac{\pi}{2^{2}}\sum_{k=0}^{\infty} \frac{(-1)^k}{\Gamma(k+3)}\left(\frac{z}{2}\right)^{2k}\sum_{r=0}^k\alpha^{2r}\beta^{2(k-r)}\frac{(2(k-r)+3)!}{r!(2(k-r))!(k-r+1)!}\notag\\
&=\frac{\pi}{2^{2}}\sum_{k=0}^{\infty} \frac{(-1)^k}{\Gamma(k+3)}\left(\frac{z}{2}\right)^{2k}\sum_{r=0}^k\alpha^{2r}\beta^{2(k-r)}\frac{(2(k-r)+3)(2(k-r)+2)(2(k-r)+1)}{r!(k-r+1)!}\notag\\
&=\frac{\pi}{2}\sum_{k=0}^{\infty} \frac{(-1)^k}{\Gamma(k+3)}\left(\frac{z}{2}\right)^{2k}\sum_{r=0}^k\alpha^{2r}\beta^{2(k-r)}\frac{(2(k-r)+3)(2(k-r)+1)}{r!(k-r)!}\notag\\
&=\frac{\pi}{2}\sum_{k=0}^{\infty} \frac{(-1)^k}{\Gamma(k+3)}\left(\frac{z}{2}\right)^{2k}\sum_{r=0}^k\alpha^{2r}\beta^{2(k-r)}\frac{4(k-r)^2+8(k-r)+3}{r!(k-r)!}\notag\\
&=\frac{\pi}{2}\sum_{k=0}^{\infty} \frac{(-1)^k}{\Gamma(k+3)}\left(\frac{z}{2}\right)^{2k}\sum_{r=0}^k\alpha^{2r}\beta^{2(k-r)}\sum_{j=0}^2\frac{a_{j,2}}{r!(k-r-j)!}\notag
\end{align*}
where $a_{0,2}=3,a_{1,2}=12,a_{2,2}=4$. Therefore, we obtain that
\begin{align}\label{eq:int2}
&\int_0^{2\pi}\exp\{iz(\alpha\cos\theta+\beta\sin\theta)\}\sin^{4}\theta d\theta\notag\\
&=\frac{\pi}{2}\sum_{k=0}^{\infty} \frac{(-1)^k}{\Gamma(k+3)}\left(\frac{z}{2}\right)^{2k}\sum_{j=0}^2\frac{a_{j,2}\beta^{2j}}{(k-j)!}\sum_{r=0}^{k-j}\binom{k-j}{r}\alpha^{2r}\beta^{2(k-j-r)}\notag\\
&=\frac{\pi}{2}\sum_{k=0}^{\infty} \frac{(-1)^k}{\Gamma(k+3)}\left(\frac{z}{2}\right)^{2k}\sum_{j=0}^2\frac{a_{j,2}\beta^{2j}}{(k-j)!}(\alpha^2+\beta^2)^{k-j}\notag\\
&=\frac{\pi}{2}\sum_{j=0}^2a_{j,2}\sum_{k=j}^{\infty}\frac{(-1)^k}{\Gamma(k+3)}\left(\frac{z}{2}\right)^{2k}\frac{\beta^{2j}}{(k-j)!}(\sqrt{\alpha^2+\beta^2})^{2(k-j)}\notag\\
&=\frac{\pi}{2}\sum_{j=0}^2a_{j,2}\beta^{2j}\sum_{k=0}^{\infty}\frac{(-1)^{k+j}}{\Gamma(k+3+j)}\left(\frac{z}{2}\right)^{2(k+j)}\frac{1}{k!}(\sqrt{\alpha^2+\beta^2})^{2k}\notag\\
&=\frac{\pi}{2}\sum_{j=0}^2(-1)^ja_{j,2}\beta^{2j}\left(\frac{z}{2}\right)^{2j}\sum_{k=0}^{\infty}\frac{(-1)^{k}}{k!\Gamma(k+3+j)}\left(\frac{z}{2}\sqrt{\alpha^2+\beta^2}\right)^{2k}\notag\\
&=2\pi\left\{\frac {3}{(z\sqrt{\alpha^2+\beta^2})^2}J_2(z\sqrt{\alpha^2+\beta^2})-\frac{6\beta^2}{z(\sqrt{\alpha^2+\beta^2})^3}J_3(z\sqrt{\alpha^2+\beta^2})+\frac{\beta^4}{(\sqrt{\alpha^2+\beta^2})^4}J_4(z\sqrt{\alpha^2+\beta^2})\right\}
\end{align}

By means of the same arguments we have that
\begin{align}\label{eq:int3}
&\int_0^{2\pi}e^{iz(\alpha\cos\theta+\beta\sin\theta)}\sin^{6}\theta d\theta\notag\\
&=\frac{\pi}{2^{2}}\sum_{k=0}^{\infty} \frac{(-1)^k}{\Gamma(k+n+1)}\left(\frac{z}{2}\right)^{2k}\sum_{r=0}^k\alpha^{2r}\beta^{2(k-r)}\sum_{j=0}^3\frac{a_{j,3}}{r!(k-r-j)!}\notag\\
&=2\pi\Bigg\{\frac {15}{(z\sqrt{\alpha^2+\beta^2})^3}J_3(z\sqrt{\alpha^2+\beta^2})-\frac {45\beta^2}{z^2(\sqrt{\alpha^2+\beta^2})^4}J_4(z\sqrt{\alpha^2+\beta^2})+\frac {15\beta^4}{z(\sqrt{\alpha^2+\beta^2})^5}J_5(z\sqrt{\alpha^2+\beta^2})\notag\\
&\quad-\frac {\beta^6}{(\sqrt{\alpha^2+\beta^2})^6}J_6(z\sqrt{\alpha^2+\beta^2})\Bigg\}
\end{align}
where $a_{0,3}=15,a_{1,3}=90,a_{2,3}=60,a_{3,3}=8$.

For $\nu=n$, with $n\geq 0,$ one has that
\begin{align*}
&\int_0^{2\pi}e^{iz(\alpha\cos\theta+\beta\sin\theta)}\sin^{2n}\theta d\theta\notag\\
&=4\sum_{k=0}^{\infty} \frac{(-1)^k}{(2k!)}z^{2k}\sum_{r=0}^k\binom{2k}{2r}\alpha^{2r}\beta^{2(k-r)}\int_0^{\pi/2}\cos^{2r}\theta\sin^{2(k-r+n)}\theta d\theta\notag\\
&=2\pi\sum_{k=0}^{\infty} \frac{(-1)^k}{(2k!)}z^{2k}\sum_{r=0}^k\binom{2k}{2r}\alpha^{2r}\beta^{2(k-r)}\frac{\Gamma(r+1/2)\Gamma(k-r+n+1/2)}{\Gamma(k+n+1)}\notag\\
&=2\pi\sum_{k=0}^{\infty} \frac{(-1)^k}{(2k!)}\frac{z^{2k}}{\Gamma(k+n+1)}\sum_{r=0}^k\frac{(2k)!}{(2r)!(2(k-r))!}\alpha^{2r}\beta^{2(k-r)}\frac{\Gamma(2r)\Gamma(2(k-r+n))2^{2-2(k+n)}}{\Gamma(r)\Gamma(k-r+n)}\notag\\
&=\frac{\pi}{2^{2n-2}}\sum_{k=0}^{\infty} \frac{(-1)^k}{\Gamma(k+n+1)}\left(\frac{z}{2}\right)^{2k}\sum_{r=0}^k\alpha^{2r}\beta^{2(k-r)}\frac{(2(k-r)+2n-1)!}{r!(2(k-r))!(k-r+n-1)!}\notag\\
&=\frac{\pi}{2^{2n-2}}\sum_{k=0}^{\infty} \frac{(-1)^k}{\Gamma(k+n+1)}\left(\frac{z}{2}\right)^{2k}\sum_{r=0}^k\alpha^{2r}\beta^{2(k-r)}\frac{(2(k-r)+2n-1)(2(k-r)+2n-2)\cdots(2(k-r)+1)}{r!(k-r+n-1)!}\notag\\
&=\frac{\pi}{2^{n-1}}\sum_{k=0}^{\infty} \frac{(-1)^k}{\Gamma(k+n+1)}\left(\frac{z}{2}\right)^{2k}\sum_{r=0}^k\alpha^{2r}\beta^{2(k-r)}\frac{(2(k-r)+2n-1)(2(k-r)+2n-3)\cdots(2(k-r)+1)}{r!(k-r)!}\notag
\end{align*}

Therefore, by taking into account  similar manipulations appearing in \eqref{eq:int}, \eqref{eq:int2} and \eqref{eq:int3}, we can infer that the following equality holds
$$\frac{(2(k-r)+2n-1)(2(k-r)+2n-3)\cdots(2(k-r)+1)}{r!(k-r)!}=\sum_{j=0}^n\frac{a_{j,n}}{r!(k-r-j)!}$$
where $a_{j,n}$s are positive constants. We are not able to explicit $a_{j,n}$ for all values of $n$ and $j$ because  for the above constants seem that a recursive rule does not hold. Nevertheless, it is easy to see that 
$$a_{0,n}=(2n-1)(2n-3)\cdots3\cdot1,\quad a_{n,n}=2^n$$

Furthermore, the cases tackled above can be sum up in the following scheme

\begin{align}\label{const}
&n=0 \to a_{0,0}=1\notag\\
&n=1\to a_{0,1}=1,\,a_{1,1}=2 \notag\\
&n=2\to a_{0,2}=3,\,a_{1,2}=12,\,a_{2,2}=4\\
&n=3\to a_{0,3}=15,\,a_{1,3}=90,\,a_{2,3}=60,\,a_{3,3}=8 \notag
\end{align}

Finally, we get that

\begin{align}\label{eq:intgeneral}
&\int_0^{2\pi}e^{iz(\alpha\cos\theta+\beta\sin\theta)}\sin^{2n}\theta d\theta\notag\\
&=\frac{\pi}{2^{n-1}}\sum_{k=0}^{\infty} \frac{(-1)^k}{\Gamma(k+n+1)}\left(\frac{z}{2}\right)^{2k}\sum_{r=0}^k\alpha^{2r}\beta^{2(k-r)}\sum_{j=0}^n\frac{a_{j,n}}{r!(k-r-j)!}\notag\\
&=\frac{\pi}{2^{n-1}}\sum_{k=0}^{\infty} \frac{(-1)^k}{\Gamma(k+n+1)}\left(\frac{z}{2}\right)^{2k}\sum_{j=0}^n\frac{a_{j,n}\beta^{2j}}{(k-j)!}\sum_{r=0}^{k-j}\binom{k-j}{r}\alpha^{2r}\beta^{2(k-j-r)}\notag\\
&=\frac{\pi}{2^{n-1}}\sum_{k=0}^{\infty} \frac{(-1)^k}{\Gamma(k+n+1)}\left(\frac{z}{2}\right)^{2k}\sum_{j=0}^n\frac{a_{j,n}\beta^{2j}}{(k-j)!}(\alpha^2+\beta^2)^{k-j}\notag\\
&=\frac{\pi}{2^{n-1}}\sum_{j=0}^na_{j,n}\sum_{k=j}^{\infty}\frac{(-1)^k}{\Gamma(k+n+1)}\left(\frac{z}{2}\right)^{2k}\frac{\beta^{2j}}{(k-j)!}(\sqrt{\alpha^2+\beta^2})^{2(k-j)}\notag\\
&=\frac{\pi}{2^{n-1}}\sum_{j=0}^na_{j,n}\beta^{2j}\sum_{k=0}^{\infty}\frac{(-1)^{k+j}}{\Gamma(k+n+j+1)}\left(\frac{z}{2}\right)^{2(k+j)}\frac{1}{k!}(\sqrt{\alpha^2+\beta^2})^{2k}\notag\\
&=\frac{\pi}{2^{n-1}}\sum_{j=0}^n(-1)^ja_{j,n}\beta^{2j}\left(\frac{z}{2}\right)^{2j}\sum_{k=0}^{\infty}\frac{(-1)^{k}}{k!\Gamma(k+n+j+1)}\left(\frac{z}{2}\sqrt{\alpha^2+\beta^2}\right)^{2k}\notag\\
&=\frac{\pi}{2^{n-1}}\sum_{j=0}^n(-1)^ja_{j,n}\beta^{2j}\frac{(z/2)^{-n+j}}{(\sqrt{\alpha^2+\beta^2})^{n+j}}J_{n+j}(z\sqrt{\alpha^2+\beta^2})\notag\\
&=2\pi\sum_{j=0}^n\frac{(-1)^ja_{j,n}\beta^{2j}}{2^{j}z^{n-j}(\sqrt{\alpha^2+\beta^2})^{n+j}}J_{n+j}(z\sqrt{\alpha^2+\beta^2})
\end{align}

Important results are also the following ones
\begin{equation}\label{formula1}
\int_0^ax^\mu(a-x)^\nu J_\mu(x)J_\nu(a-x)dx=\frac{\Gamma(\mu+\frac12)\Gamma(\nu+\frac12)}{\sqrt{2\pi}\Gamma(\mu+\nu+1)}a^{\mu+\nu+\frac12}J_{\mu+\nu+\frac12}(a),
\end{equation}
with $Re\,\mu>-\frac12$ and $Re\,\nu>-\frac12$ (see Gradshteyn-Ryzhik, 1980, formula
6.581(3)),
\begin{equation}\label{formula2}
\int_0^a\frac{J_\mu(x)}{x}\frac{J_\nu(a-x)}{a-x}dx=\left(\frac1\mu+\frac1\nu\right)\frac{J_{\mu+\nu}(a)}{a},
\end{equation}
with $Re(\mu)>0,\,Re(\nu)>0$ (see Gradshteyn and Ryzhik, 1980, formula 6.533(2)), and

 \begin{equation}\label{formula3}
\int_0^\infty  x^{\mu-\nu}J_{\nu+1}(ax)J_\mu(bx)dx=\frac{(a^2-b^2)^{\nu-\mu}b^\mu}{2^{\nu-\mu}a^{\nu+1}\Gamma(\nu-\mu+1)},
\end{equation}
with $a\geq b$, $Re(\nu+1)>Re(\mu)>0$ (see Gradshteyn and Ryzhik, 1980, formula 6.575(1)),

\begin{equation}\label{formula4}
\int_0^{\pi/2}(\sin x)^{\nu+1}\cos(b\cos x)J_\nu(a\sin
x)dx=\sqrt{\frac{\pi}{2}}\frac{a^\nu J_{\nu+\frac{1}{2}}\left(\sqrt{a^2+b^2}\right)}{(a^2+b^2)^{\frac{\nu}{2}+\frac{1}{4}}}
,
\end{equation}
for $Re~\nu>-1$ (see Gradshteyn-Ryzhik, 1980, pag. 743, formula
6.688.(2)).

\end{document}